\documentclass[12pt]{amsart}
\usepackage{amsmath,amsfonts,amssymb,mathrsfs,amstext,amscd,latexsym,amsthm}
\usepackage[margin=1.2in]{geometry}
\usepackage{comment}
\usepackage{mathtools}
\usepackage{hyperref}
\usepackage{enumitem}
\usepackage{xcolor}
\newcommand{\R}{\ensuremath{\mathbb{R}}}

\DeclareMathOperator{\arctanh}{\mathrm{arctanh}} 
\DeclareMathOperator{\arcsinh}{\mathrm{arcsinh}} 
\DeclareMathOperator{\sech}{\mathrm{sech}}

\DeclareMathOperator{\width}{\mathrm{width}}

\DeclareMathOperator{\tr}{\mathrm{tr}}

\newcommand{\pd}{\partial}

\DeclareMathOperator{\area}{\mathrm{area}} 
\DeclareMathOperator{\length}{\mathrm{length}} 
\DeclareMathOperator{\Rc}{\mathrm{Rc}} 
\DeclareMathOperator{\Rm}{\mathrm{Rm}} 
\DeclareMathOperator{\Sc}{\mathrm{R}} 

\def\labelitemi{--}
\def\ba #1\ea {\begin{align} #1\end{align}}
\def\bann #1\eann {\begin{align*} #1\end{align*}}
\def\ben #1\een {\begin{enumerate} #1\end{enumerate}}
\def\bi #1\ei {\begin{itemize}\renewcommand\labelitemi{--} #1\end{itemize}}
\theoremstyle{plain}
\newtheorem{thm}{Theorem}[section]
\newtheorem*{thm*}{Theorem}
\newtheorem{lem}[thm]{Lemma}

\newtheorem{prop}[thm]{Proposition}
\newtheorem{conj}[thm]{Conjecture}
\newtheorem{defn}[thm]{Definition}

\theoremstyle{remark}
\newtheorem{rem}{Remark}

\usepackage{cite}

\author[T.~Bourni]{Theodora Bourni}
\author[T.~Buttsworth]{Timothy Buttsworth}
\author[R.~Lafuente]{Ramiro Lafuente}
\author[M.~Langford]{Mat Langford}

\date{\today}

\title{Ancient Ricci flows of bounded girth}

\begin{document}

\begin{abstract}
For each $n\ge 3$, we construct a `pancake-like', $O(2)\times O(n-1)$-invariant ancient Ricci flow 
with positive curvature operator and bounded ``girth'', and we determine its asymptotic limits backwards in time. This solution is new even in dimension three.  
The construction hinges on the Ricci flow invariance of certain  conditions  on the curvature and its spatial derivatives under this symmetry regime, whose proof does not follow from Hamilton's tensor maximum principle. 
\end{abstract}


\maketitle

\tableofcontents

\section{Introduction}


A solution to the Ricci flow equation $\tfrac{\partial}{\partial t} g(t) = - 2\, \rm{Ric}_{g(t)}$ is called \emph{ancient} if it is defined for all $t\in (-\infty, T)$. Ancient solutions to geometric flows have aroused a great deal of interest in recent years, as they are natural parabolic counterparts of complete solutions to the corresponding elliptic equation; see \cite{MR4204997} for a recent survey. We are particularly interested in the Liouville-type rigidity phenomena they may exhibit, in the spirit of Appell's theorem \cite{MR49430} stating that positive ancient solutions to the heat equation on $\R^n$ with sub-exponential growth must be constant.

With Appell's theorem in mind, it is natural to consider ancient Ricci flows with positive curvature operator. In fact, this is not actually a restriction, at least in low dimensions, since Ricci flow tends to force curvature towards the positive\cite{MR3997128,MR3207356,MR1714939,MR1249376}. In the compact case, positivity of the curvature operator forces the underlying manifold to be $S^n$, up to a quotient \cite{BohmWilking}. Ancient Ricci flows on $S^2$ are completely classified \cite{DHSRicci}, the only examples being the shrinking round spheres and the ``ancient sausage'' solution\footnote{Discovered independently by Fateev--Onofri--Zamolodchikov \cite{Fateev0}, King \cite{MR1214546} and Rosenau \cite{Rosenau}.}.  Remarkably, in a recent major achievement, positively-curved ancient Ricci flows on $S^3$  have also been classified, under an additional  non-collapsing assumption \cite{MR4323639}: the only examples are the shrinking round spheres and Perelman's `football' solution \cite{Perelman2}.  This classification has been extended to higher dimensions $n\geq 4$, under suitable curvature positivity assumptions \cite{RFuniquenessHighDim}. We note here that the non-collapsing assumption, however natural from the singularity formation viewpoint (due to Perelman's ``no local collapsing'' theorem \cite{Perelman1}), is not as natural in the compact case (indeed, any compact blow-up limit of a finite-time Ricci flow singularity must be a shrinking soliton \cite{ZHANG2007503}). Nonetheless, given the difficulty of the problem, these results are a remarkable achievement and provide spectacular progress towards a general classification of ancient Ricci flows in low dimensions.

Our main result in this paper is the construction of an interesting new ancient Ricci flow on $S^n$ with positive curvature operator in each dimension $n\geq 3$.

\begin{thm}\label{mainexistence}
For each $n\geq 3$, there exists an $O(2)\times O(n-1)$-invariant ancient Ricci flow on ${S}^n$ with positive curvature operator and bounded \emph{girth}.  Pointed limits of this solution are asymptotic, as $t\to -\infty$, to either a flat cylinder $S^1 \times \R^{n-1}$ or a cigar plane $(\R^2, g_{\sf cigar}) \times \R^{n-2}$. 
\end{thm}

Roughly speaking, the \emph{girth} is the length of a shortest closed geodesic; see \S \ref{sec:girth}  for the precise definition. The bounded girth condition  implies that the solution is collapsed. The group $O(2)\times O(n-1) \subset O(n+1)$ (standard block embedding) acts on  $\R^{n+1} \supset S^n$ by orthogonal transformations, with generic orbits of codimension one when restricted to $S^n$. We call the fixed-point sets of the $O(2)$- and $O(n-1)$- actions the \emph{tip} and the \emph{waist}, respectively. When the sequence of marked points stays within bounded distance from the \emph{waist}, the asymptotic limit is the flat cylinder, whereas if they stay close to the \emph{tip}, the limit is a cigar plane. 
We note that the tip and waist regions are both of the same scale: the asymptotic cylinder at the waist and the asymptotic cylinder of the asymptotic cigar both have the same radius. 

Geometrically, our  terminology for the fixed-point sets is justified by considering, for each time slice $(S^n,g(t))$, a totally geodesic $S^2$ intersecting each $O(n-1)$-orbit orthogonally. Such a surface, which we refer to as the \emph{section}, exists in fact for any $O(2)\times O(n-1)$-invariant metric on $S^n$ (indeed, the $O(n-1)$-action is polar). In the case of our solution, it is a highly elongated, rotationally symmetric two-sphere, whose equator (the waist) realises the girth, which is uniformly bounded as $t\to -\infty$, and whose poles are the so-called tips. Our solution may be thought of as `rotating' these elongated two-sphere while keeping their waist fixed; if instead of two-spheres we had highly elongated ovals, our solution in dimension 3 would resemble a pancake, with bounded thickness but diverging radius as $t\to -\infty$.

The proof of Theorem \ref{mainexistence} is in part inspired by the construction of ancient mean curvature flows in \cite{BLT1} (cf.~\cite{MR4109899}).  We consider a sequence of old-but-not-ancient $O(2)\times O(n-1)$-invariant Ricci flow solutions on $S^n$ whose initial data has a section $S^2$ which is a time-slice of the ancient sausage solution. The estimates required to extract a limit, to estimate the existence time, and to analyse the asymptotic behavior, rely crucially on a rather precise space-time control of the four different eigenvalues of the curvature operator; see Proposition \ref{prop:curvaturepreserved} (we note here that the second fundamental form in the mean curvature flow solution \cite{BLT1} has only two different eigenvalues). Unfortunately, Hamilton's tensor maximum principle is not applicable in this situation, and we must rely on a subtle analysis of the PDE system satisfied by the curvature operator under this symmetry assumption, which has unbounded coefficients due to the presence of lower-dimensional orbits.

There do not appear to be many other known examples of collapsed ancient Ricci flows with positive curvature operator. In dimension 2, a classification of \textit{all} complete ancient Ricci flows is possible: \cite{Chowetal,DHSRicci,MR2264733,MR954419,MR2301253}; all of them are collapsed with positive curvature operator, apart from those with constant sectional curvature. On $S^3$, we know that any ancient solution on $S^3$ has positive curvature (non-negativity follows from \cite{10.4310/jdg/1246888488}, and positivity by the strong maximum principle), and yet the classification result of \cite{BrendleHuiskenSinestrari} implies that this positive curvature cannot be pinched for non-round solutions. To our knowledge, an exhaustive list of known collapsed examples on $S^3$ consists of Fateev's ``ancient hypersausage'' \cite{Fateev1,Fateev2}, the homogeneous ``ancient Hopf fibrations''  
of Bakas, Kong and Ni \cite{BakasKongNi} 
(see also \cite{Buz14,CaoSaloff-Coste,IJ92,Lau13}), and a family interpolating between them, which we shall call the ``twisted ancient hypersausages'' \cite[\S 3]{BakasKongNi}. In the non-compact setting, Y.~Lai was able to construct, in each dimension $n\ge 3$, a family of ``flying wing'' steady Ricci solitons on $\R^n$, which are $O(n-1)$-invariant and interpolate between a cigar-hyperplane and Bryant's soliton \cite{Lai2,Lai1}; our construction is in some ways analogous to Lai's: in both cases, we are solving a geometric partial differential equation on a non-compact domain (hers elliptic and ours parabolic) with a symmetry assumption which leaves two degrees of freedom (subverting any attempt at reduction to a system of \textsc{ode}), and in neither case are the powerful tools related to non-collapsing available.

On the other hand, much is known about ancient Ricci flows that are \textit{non-collapsed}, at least in the presence of various curvature non-negativity assumptions. For example, non-collapsed ancient Ricci flows with uniform PIC and weak PIC2 have been classified \cite{MR4400906,MR4176064,MR4323639}; they consist of the shrinking spheres, the Bryant solitons, as well as Perelman's ``football'' solutions. If we maintain weak PIC2 and non-collapsing, and replace uniform PIC with Type-I curvature growth, the ancient Ricci flow is a symmetric space  \cite{LynchAbrego}. On the other-hand, if we drop the Type-I curvature growth condition, but keep non-negative curvature operator and non-collapsing, there appear to be many more: analogous solutions to Perelman's football \cite{Perelman2} in higher dimensions with $O(k)\times O(n+1-k)$-symmetry (see \cite{MR2604955,MR2365237} for $k=1$, and \cite{Buttsworth} for the case $k=2, n=4$); and also 
Haslhofer's new solutions \cite{Haslhofer4d}. 

We also mention some results on homogeneous ancient Ricci flows, i.e., those ancient solutions which admit a transitive and isometric group action. Homogeneous ancient solutions on spheres of any dimension have been recently classified by Sbiti \cite{Sb22}, and it appears that only the round sphere has positive curvature operator. 
Further examples on homogeneous torus bundles were constructed by Lu--Wang \cite{MR3689745}, Buzano \cite{Buz14} and B\"ohm--Lafuente--Simon \cite{MR3984075}. In fact, a compact homogeneous space admits a collapsing ancient homogeneous Ricci flow (of not necessarily positive curvature) if and only if it is the total space of a homogeneous torus bundle; all known homogeneous examples are invariant under a corresponding torus action (this is known to be necessary under certain assumptions \cite{KrishnanPediconiSbiti}) and, after appropriately rescaling, collapse the torus fibres as time tends to minus infinity and Gromov--Hausdorff converge to an Einstein metric on the base \cite{CaoSaloff-Coste,PediconiSbiti}.

Finally, despite the fact that there are clearly many ancient Ricci flows --- as indicated in the previous paragraphs --- we conjecture that our example is unique amongst $O(n-1)$-invariant ancient Ricci flows on $S^n$, $n\ge 4$, with positive curvature operator and bounded girth
.

\begin{conj}\label{mainconjecture}
Let $n\ge 4$. Then, up to isometry, parabolic rescalings and time-translation, the Ricci flow solution on $S^n$ constructed in Theorem \ref{mainexistence} is the unique $O(n-1)$-invariant ancient solution with  positive curvature operator and bounded girth.
\end{conj}

\begin{rem}
Note that we do not assume the full $O(2)\times O(n-1)$-symmetry in this conjecture. Note also that the conjecture fails in dimension three due to the existence of the (torus fibred) ancient hypersausage. (The $n=3$ case is special is because $O(2)$ is Abelian.)
\end{rem}

\subsection*{Acknowledgements}

TBu, RL and ML wish to thank MATRIX for providing the stimulating climate which allowed this project to take off. TBu also wishes to thank Wolfgang Ziller for his interest in the project, as well as many insightful discussions about the relevant group actions. RL would like to thank Christoph Böhm for sharing his insight into Hamilton's maximum principle. TBo was supported by the Simons Foundation (grant 707699) and the National Science Foundation (grant DMS-2105026). TBu, RL and ML were supported by the Australian Research Council (grants DE220100919, DE190101063 and DE200101834, respectively).

\section{Preliminaries}\label{sec:prelims}

We begin by developing a precise notion of ``girth" suitable for Theorem \ref{mainexistence} and Conjecture~\ref{mainconjecture}. We then describe the class of Riemannian metrics on ${S}^n$ that we work with, and collect some basic facts about how these metrics behave under the Ricci flow.



\subsection{The girth of a Riemannian sphere}\label{sec:girth}

Given a smooth manifold $M$, let $\Lambda M$ be the space of piecewise smooth maps $S^1 \to M$, and let $\Lambda^0 M \subset \Lambda M$ denote the constant maps. Recall that the energy and length functionals $E, \ell: \Lambda M \to \mathbb{R}$, are defined by
\[
    E(c) := \tfrac12 \int_{S^1} |c'|^2 d\theta, \qquad \length(c) := \int_{S^1} |c'| \, d\theta.
\]
Observe that $\length(c) \leq \sqrt{2 E(c)}$, with equality when the parametrisation is proportional to arclength. Consider the homotopy class $\Sigma_{1}$ of continuous maps
\[
    F: B^{n-1} \to \Lambda M, \qquad F(\partial B^{n-1}) \subset \Lambda^0 M
\]
which have degree 1, in a sense to be made precise below. Elements of $\Sigma_1$ are called \emph{$1$-sweepouts}.

\begin{defn}
The \emph{$1$-width} and \emph{girth} of $(S^n,g)$ are defined by
\[
          {\rm width}_1(S^n,g) :=   \inf_{F\in \Sigma_{1}} \max_{p\in B^{n-1}} E(F(p)), \qquad  {\rm girth}(S^n,g) := \sqrt{2 \, \width_1(S^n,g)}.
\]
\end{defn}

It follows from \cite[Appendix 1]{Klin78} that these invariants are positive and in fact realised by a closed geodesic in $(S^n,g)$ (indeed, \cite[Lemma A.1.4]{Klin78} provides a homotopy which decreases the maximum energy of a sweepout $F$ if there are no closed geodesics of that energy level). 

Following \cite{Klin78}, we now explain how a  continuous map
\[
    F: B^{n-1} \to \Lambda M, \qquad F(\partial B^{n-1}) \subset \Lambda^0 M
\]
induces a continuous map $f: S^n \to S^n$, which we use to define the degree $\deg F := \deg f$ of the sweepout $F$. Using coordinates $(x_0, \ldots, x_n)$ for $\mathbb{R}^{n+1}$, we view $B^{n-1}$ as a half-equator 
\[
  B^{n-1} \simeq \{ x_0 \geq 0, x_1 = 0 \} \subset S^n \subset \mathbb{R}^{n+1}.
\]
For each $p\in B^{n-1}$ there is a circle $\theta\mapsto a_p(\theta)$ in $S^n$, $\theta\in S^1$, which starts at $a_p(1) = p$ in the direction of $x_1 \geq 0$, and along which the coordinates $x_2 \ldots, x_n$ remain constant. Notice that for $p\in \partial B^{n-1} \simeq \{ x_0 = x_1 = 0 \}$, $a_p(\theta) \equiv p$ is constant. Thus, any $q\in S^n$ admits a representation $q = a_p(\theta)$, for a unique $p\in B^{n-1}$, and $\theta\in S^1$ which is unique when $p\notin \partial B^{n-1}$. We then simply set
\[
 f(a_p(\theta)) :=    F(p)(\theta). 
\]
Importantly, it is easy to see that a homotopy  $[0,1]\times (B^{n-1}, \partial B^{n-1}) \to (\Lambda M, \Lambda^0 M)$  of $F$ induces a homotopy $[0,1] \times S^n \to S^n$ of $f$.  

\begin{rem}
The analogous notion of $2$-width  is known to interact well with the Ricci flow. Indeed, Colding and Minicozzi \cite{ColdingMinicozziWidth} obtained a differential inequality for the evolution of the $2$-width of any homotopy $3$-sphere under Ricci flow-with-surgery, which they were able to exploit to provide a different proof that the flow-with-surgery terminates in finite time in this setting (see also Hamilton \cite{MR1714939} and Perelman \cite{Perelman3}). 
\end{rem}
  
\begin{rem}
The above notion of width originates in work of Birkhoff \cite{Birkhoff}. A number of related (but not necessarily equivalent) notions of width have also been developed, including the Almgren--Pitts width, which is based on Almgren's isomorphism theorem and was exploited in the celebrated Almgren--Pitts min-max theory \cite{Almgren,Pitts} (see also \cite{MarquesNevesSurvey} for a survey), and Gromov's definition in terms of level sets of proper functions $M^n \to \mathbb{R}^{n-1}$ \cite{Gromov}. The later was invoked in the study of Ricci flows on surfaces by Daskalopoulos--Hamilton \cite{MR2074874}  and Daskalopoulos--\v{S}e\v{s}um \cite{MR2264733} (see also Chu \cite{MR2301253}). For the purposes of this paper, these notions of 1-width can be used in place of the 1-girth defined above. We have chosen to use the latter, however, as it easier to work with (e.g. it is always attained by the length of some smooth, closed geodesic).
\end{rem}

\subsection{$O(n-1)$-invariant Riemannian metrics}\label{on-1} Consider the standard action of $O(n-1)$ on $S^n \subset \mathbb{R}^{n+1}$ by rotating the last $n-1$ coordinates, and let $g$ be any $O(n-1)$-invariant Riemannian metric on ${S}^n$, i.e.~ one for which  $O(n-1)$ acts by isometries. The image $\Sigma \subset S^n$ of the embedding $\gamma:{S}^2\to {S}^{n}$,
\begin{align*}
\gamma(x_1,x_2,x_3)=(x_1,x_2, 0,\cdots,0, x_3),
\end{align*}
is a totally-geodesic submanifold of $(S^n,g)$, as it co-incides with the fixed point set of the closed subgroup $O(n-2)\subset O(n-1)$ rotating the $n-2$ coordinates $x_3, \ldots, x_{n}$, which is non-trivial for $n\geq 3$. Similarly, the image $\gamma(E)$ of the equator $E=\{(x_1,x_2,0)\in {S}^2\}$ is a closed geodesic, being the fixed point set of the full group $O(n-1)$.  Given any $p\in S^n\setminus \gamma(E)$, its $O(n-1)$-orbit is an $(n-2)$-dimensional sphere, intersecting $\Sigma$ exactly twice. We thus define $\Sigma_+ := \gamma(S_+^2)$
where $S_+^2 = \{ x_3 > 0\}$ denotes the northern hemisphere. It follows that $\Sigma_+$ intersects each $O(n-1)$-orbit exactly once, and $\partial \Sigma_+ = \gamma(E)$ is a closed geodesic.

For each point $x\in{S}^2\setminus E$ we write 
\begin{align}\label{ch2decomp}
T_{\gamma(x)}{S}^n=T_{\gamma(x)}\gamma({S}^2)\oplus K,
\end{align}
where $K$ is the subspace of $\mathbb{R}^{n+1}$ consisting of points $x$ where the first, second and last entries are $0$ (recall that $n+1 \geq  4$, thus $K\neq 0$). The Euclidean inner product on $\mathbb{R}^{n+1}$ induces an $O(n-2)$-invariant inner product $Q$ on $K$. 

\begin{lem}
Let $g$ be an $O(n-1)$-invariant Riemannian metric on ${S}^{n}$, $n\geq 3$. Along $\Sigma_+$, $g$ can be written as a warped product
\begin{align}\label{ch2metricform}
g = g^\top + \varphi^2 g_{S^{n-2}},
\end{align}
where $g^\top = g|_{T\Sigma}$, $g_{S^{n-2}}$ is the round metric on $S^{n-2}$, and $\varphi : \Sigma^+ \to \mathbb{R}_{>0}$ is smooth.
\end{lem}
\begin{proof}
For non-equator points $p\in \Sigma_+$, the action of the isotropy subgroup $O(n-2)$ on $T_p{S}^n = T_p \Sigma \oplus K$ is trivial on $T_p \Sigma$ and  contains no trivial subspaces in $K$ (even if $n=3$, where the action is simply a reflection). By Schur's lemma it follows that $g$ makes \eqref{ch2decomp} orthogonal, and that restricted to $K$ it must be a multiple of $Q$.   
\end{proof}
\begin{rem}
The converse of this lemma is not  true. In order for the expression in \eqref{ch2metricform} to define a Riemannian metric $g$ which extends smoothly to all of ${S}^n$ it is necessary that 
\[
	\bar\varphi(x_1,x_2,x_3) :=\begin{cases}	
		\varphi(\gamma(x_1,x_2,x_3)), \qquad &x_3 > 0;\\ 
		0, \qquad  &x_3 = 0;\\
		- \varphi(\gamma(x_1,x_2,x_3)), \qquad &x_3 < 0,
	\end{cases}
\]
is a smooth function on $S^2$.
\end{rem}

\subsection{$O(2)\times O(n-1)$-invariant Riemannian metrics}\label{sec:cohom1_metrics}
The natural action of $O(2)\times O(n-1)$ on $S^n\subset \mathbb{R}^2\oplus \mathbb{R}^n$ is of cohomogeneity one. In the coordinates described below, the quotient map of this action is given by $r : S^n \to [0,\pi/2]$. The principal orbits $r^{-1}(r_0)$, $r_0\in(0,\pi/2)$, are diffeomorphic to $S^1\times S^{n-2}$. The singular orbit $r^{-1}(0) \simeq S^1$ will be called the \emph{waist}, and the singular orbit $r^{-1}(\pi/2) \simeq S^{n-2}$ will be referred to as the \emph{tip}. This terminology is justified when considering the 2-dimensional picture obtained by quotienting out the $O(n-1)$-action.

To describe $O(2)\times O(n-1)$-invariant Riemannian metrics on ${S}^n$, we define  
\[
	\gamma:[0,\tfrac{\pi}{2}]\to {S}^{n}, \qquad \gamma(r)=(\cos(r),0,\ldots,0,\sin(r)).
\]
This curve starts at the waist, ends at the tip, and intersects all the orbits  exactly once. Moreover, the reflection about the $\langle e_1, e_{n+1}\rangle$-plane in $\mathbb{R}^{n+1}$ is an element in $O(2)\times O(n-1)$ which pointwise fixes the curve, thus $\gamma$ is a geodesic for any invariant metric. The principal domain for the action (i.e.~the union of all orbits through $\gamma(r)$ for $r\in (0,\pi/2)$) is diffeomorphic to $(0,\pi/2) \times S^1 \times S^{n-2}$. We denote by $\theta$ the standard coordinate on $S^1$ and by $g_{S^{n-2}}$ the round metric on $S^{n-2}$. Also, given a smooth function $f: [0,\pi/2] \to \mathbb{R}$, we call it \emph{even} (resp.~\emph{odd}) \emph{at $0$}, if its even (resp.~odd) extension to $(-\pi/2, \pi/2)$ is smooth. Analogously terminology will be used at $\pi/2$.     We then have:



\begin{prop}\label{prop:ch1mfsmoothness}
For any $O(2)\times O(n-1)$-invariant Riemannian metric $g$ on ${S}^n$, $n\geq 3$ there exist smooth functions $\chi,\psi,\varphi:[0,\frac{\pi}{2}]\to \mathbb{R}$, positive on $(0,\pi/2)$, so that 
\begin{align}\label{ch1mf}
g_{\gamma(r)}=\chi(r)^2 dr^2 +\psi(r)^2 d\theta^2 +\varphi(r)^2 g_{S^{n-2}}, \qquad \hbox{for all } r\in (0,\pi/2),
\end{align}
while at $r=0,\frac{\pi}{2}$, the following `smoothness conditions' are satisfied:
\begin{enumerate}
\item $\chi$ is even and positive at $0$ and $\pi/2$;
\item $\psi$ is even and positive at $0$ and odd at $\pi/2$, and $\psi'(\tfrac{\pi}{2}) = -\chi(\tfrac{\pi}{2})$;
\item $\varphi$ is even and positive at $\pi/2$ and odd at $0$, and $\varphi'(0) = \chi(0)$.
\end{enumerate}
Conversely,  any smooth functions $\chi,\psi,\varphi$ satisfying (1)-(3) give rise to a metric via \eqref{ch1mf} which extends to a smooth, $O(2)\times O(n-1)$-invariant Riemannian metric on ${S}^n$. 
\end{prop}

\begin{proof}
By considering the reflections through the subspaces $\langle e_2\rangle, \langle e_3, \ldots, e_n \rangle \subset \mathbb{R}^{n+1}$ it  follows that any invariant metric makes the factors $(0,\pi/2)$, $S^1$ and $S^{n-2}$ orthogonal. Since its restriction to $S^{n-2}$ is $O(n-1)$-invariant, said restriction must be a multiple of $g_{S^{n-2}}$ and hence \eqref{ch1mf} follows. Regarding the smoothness conditions, these are well-known and we refer the reader to \cite[Ch.~4]{Petersen} for the details.
\end{proof}

Note that since these metrics are $O(n-1)$-invariant, all of the material of \S\ref{on-1} still applies. In particular, we denote by
 $g^{\top}$ the Riemannian metric on the totally geodesic submanifold $\Sigma \simeq {S}^2 \subset S^n$, which on the principal part is given by  
 \[
 	g_{\gamma(r)}^\top = \chi(r)^2 dr^2+\psi(r)^2 d\theta^2, \qquad r\in (0,\pi/2).
 \] 
 Note that $g^{\top}$ is now $O(2)$-invariant, instead of merely $O(1)$-invariant.

\subsection{$O(2)\times O(n-1)$-invariant Ricci flow on ${S}^n$}

Consider now a \emph{time-dependent} Riemannian metric $g$ on ${S}^n$ which is evolving by Ricci flow. If the initial metric is $O(2)\times O(n-1)$-invariant, the uniquess and diffeomorphism-invariance properties of the Ricci flow imply that the Riemannian metric will continue to be $O(2)\times O(n-1)$-invariant for all future times. As a consequence, on the principal domain the Riemannian metric $g$ must have the form of \eqref{ch1mf} for some functions $\chi,\psi,\varphi$ which are now time-dependent. Then, using the formula for Ricci curvature \eqref{ch1ricc}, the Ricci flow equation gives rise to a system of PDEs for the functions $\chi(r,t),\psi(r,t),\varphi(r,t)$. Introducing now the (time-dependent) arc-length parameter $s(r,t)=\int_0^r \chi(v,t)dv$, the Ricci flow equation  becomes
\begin{subequations}\label{eq:doubly warped Ricci flow}
\ba
\chi_t={}&\chi\left(\frac{\psi_{ss}}{\psi}+(n-2)\frac{\varphi_{ss}}{\varphi}\right)\label{eq:doubly warped Ricci flow chi},\\
\psi_t={}&\psi_{ss}+(n-2)\frac{\psi_s\varphi_s}{\varphi}\label{eq:doubly warped Ricci flow psi},\\
\varphi_t={}&\varphi_{ss}+\frac{\varphi_s\psi_s}{\psi}-(n-3)\frac{1-\varphi_s^2}{\varphi}\,.\label{eq:doubly warped Ricci flow phi}
\ea
\end{subequations}

Note that the commutator of $\pd_t$ and $\pd_s$ is given by
\begin{equation}\label{eq:commutator}
[\pd_t,\pd_s]
=-\left(\frac{\psi_{ss}}{\psi}+(n-2)\frac{\varphi_{ss}}{\varphi}\right)\pd_s.
\end{equation}

\subsection{Evolution of curvature under $O(2)\times O(n-1)$-invariant Ricci flow}

As described in Appendix \ref{geometry}, the Riemann curvature operator $\Rm: TS^n\wedge TS^n\to TS^n\wedge TS^n$ of the metric \eqref{ch1mf} has at most four distinct eigenvalues, and these co-incide with the following sectional curvatures:
\begin{itemize}
\item $K^{\top} = -\frac{\psi_{ss}}{\psi}$, the Gauss curvature of $({S}^2,g^{\top})$ (multiplicity one);
\item $K_1^{\perp} = -\frac{\varphi_{ss}}{\varphi}$, the sectional curvature of planes spanned by $\partial_r$ and ${S}^{n-2}$ (multiplicity $n-2$);
\item $K_2^{\perp} = -\frac{\psi_s\varphi_s}{\psi \varphi}$, the sectional curvature of planes between $\partial_{\theta}$ and ${S}^{n-2}$ (multiplicity $n-2$);
\item $L = \frac{1 - \varphi_s^2}{\varphi^2}$, the sectional curvature in the ${S}^{n-2}$ direction (only appears if $n\ge 4$; multiplicity 
$\binom{n-2}{2}$).  
\end{itemize}
Notice that even though these four quantities are only defined along principal orbits, the smoothness conditions from Proposition \ref{prop:ch1mfsmoothness} ensure  that they can be extended to smooth, invariant scalar functions on $S^n$. 

It is well-known that, with respect to a moving orthonormal frame, the curvature operator evolves under the Ricci flow according to
\begin{align*}
\frac{\partial \Rm}{\partial t}=\Delta{\Rm}+2(\Rm^2+\Rm^{\sharp}),
\end{align*}
where $\Delta$ denotes the tensor Laplacian and the operator $\cdot^\sharp$ denotes the Lie algebra square, both taken with respect to the time-varying metric $g(t)$, and $\frac{\pd}{\pd t}$ is interpreted according to the Uhlenbeck trick (see \cite[\S 2.7.2]{Chowetal} or \cite[\S 6]{AndrewsHopper}). Using the expressions for all of these terms in Appendix \ref{geometry}, we obtain the following evolution equations:
\begin{subequations}\label{eq:evolve curvatures}
\ba\label{eq:evolve Ktop}
(\pd_t-\Delta)K^\top={}&2\left((K^\top)^2+(n-2)K_1^\perp K_2^\perp\right)-2(n-2)\left(\frac{\varphi_s}{\varphi}\right)^2(K^\top-K_2^\perp),\\
(\pd_t-\Delta)K_1^\perp={}&2\left((K_1^\perp)^2+K^\top K_2^\perp+(n-3)K_1^\perp L\right)\nonumber\\
{}&+2\left(\frac{\psi_s}{\psi}\right)^2(K^\perp_2-K_1^\perp)-2(n-3)\left(\frac{\varphi_s}{\varphi}\right)^2(K^\perp_1-L)\,,\label{eq:evolve K1perp}\\
(\pd_t-\Delta)K_2^\perp={}&2\left((K_2^\perp)^2+K^\top K_1^\perp+(n-3)K_2^\perp L\right)\nonumber\\
{}&-2\left(\frac{\psi_s}{\psi}\right)^2(K_2^\perp-K_1^\perp)+2\left(\frac{\varphi_s}{\varphi}\right)^2(K^\top-K_2^\perp)\,,\;\;\text{and}\label{eq:evolve K2perp}\\
(\pd_t-\Delta)L={}&2\left((K_1^\perp)^2+(K_2^\perp)^2+(n-3)L^2\right)+4\left(\frac{\varphi_s}{\varphi}\right)^2(K^\perp_1-L)\,,\label{eq:evolve L}
\ea
\end{subequations}
where here $\Delta$ is the Laplace-Beltrami operator acting on functions. By differentiating these expressions in $s$, and using \eqref{eq:commutator}, we also obtain
\begin{subequations}\label{eq:evolve curvature gradients}
\ba\label{eq:evolve Ktop grad}
(\pd_t-\Delta)K^\top_s={}&\left(4K^\top-\left(\frac{\psi_{s}}{\psi}\right)^2-3(n-2)\left(\frac{\varphi_{s}}{\varphi}\right)^2\right)K^\top_s\nonumber\\
{}&+2(n-2)K_2^\perp(K_1^\perp)_s+2(n-2)\left(K_1^\perp+\left(\frac{\varphi_s}{\varphi}\right)^2\right)(K_2^\perp)_s\nonumber\\
{}&+4(n-2)\frac{\varphi_s}{\varphi}\left(\left(\frac{\varphi_s}{\varphi}\right)^2+K_1^\perp\right)(K^\top-K_2^\perp)
\ea
and
\ba\label{eq:evolve K2perp grad}
(\pd_t-\Delta)(K^\perp_1)_s={}&\left[4K_1^\perp+(n-3)L-3\left(\frac{\psi_s}{\psi}\right)^2-(3n-8)\left(\frac{\varphi_s}{\varphi}\right)^2\right](K_1^\perp)_s+2K_2^\perp K^\top_s\nonumber\\
{}&+2\left[K^\top+\left(\frac{\psi_s}{\psi}\right)^2\right](K_2^\perp)_s+\left[(n-3)K_1^\perp+2(n-3)\left(\frac{\varphi_s}{\varphi}\right)^2\right]L_s\nonumber\\
{}&-4\frac{\psi_s}{\psi}\left[K^\top+\left(\frac{\psi_s}{\psi}\right)^2\right](K_2^\perp-K_1^\perp)+2\frac{\varphi_s}{\varphi}\left[K_1^\perp+\left(\frac{\varphi_s}{\varphi}\right)^2\right](K_1^\perp-L)\,.
\ea
\end{subequations}



\section{Existence}

We will construct an $O(2)\times O(n-1)$-invariant ancient Ricci flow by taking a limit of a family of ``very old'' Ricci flow solutions which start at symmetric extensions to $S^n$ by ``very early'' time slices of the $O(2)$-invariant ancient sausage Ricci flow on ${S}^2$ (see Appendix \ref{app_KR}).

\subsection{The initial metrics}
For each $\tau<0$, consider the functions 
\begin{align}\label{initialfunctions}
\begin{split}
\chi_{\tau}(r)&=\sqrt{\frac{\tanh(-2\tau)}{1-\sin^2(r)\tanh^2(-2\tau)}},\\
\psi_{\tau}(r)&=\chi_{\tau}(r)\cos(r),\\
\varphi_{\tau}(r)&=\frac{\arctanh(\tanh(-2\tau)\sin(r))}{\sqrt{\tanh(-2\tau)}}.
\end{split}
\end{align}
It is straightforward to verify that these functions satisfy the boundary smoothness conditions of Proposition \ref{prop:ch1mfsmoothness}, so the expression
\begin{align}\label{eqn_g_tau}
	\underline g_{\tau} := \chi_\tau^2 \, dr^2 + \psi_\tau^2 \, d\theta^2 + \varphi_\tau^2 \, g_{S^{n-2}}.
\end{align}
defines a smooth $O(2)\times O(n-1)$-invariant Riemannian metric on $S^n$.

Notice that 
\[
	\chi_{\tau}^2 \, dr^2+\psi_{\tau}^2  \, d\theta^2 = g_{\textrm{sausage}}(\tau)
\]
is a time slice of the ancient sausage solution on $S^2$ (see Appendix \ref{app_KR}), whereas $\varphi_{\tau}$ has been chosen so that $\frac{\varphi_{\tau}'(r)}{\psi_{\tau}(r)\chi_{\tau}(r)} = C$ is constant in $r$ (the value of $C$ is determined by $\chi_{\tau},\psi_{\tau}$ and the smoothness conditions from Proposition \ref{prop:ch1mfsmoothness}). This in particular yields that the sectional curvatures of $\underline g_\tau$ satisfy $K_{1}^{\perp}=K_{2}^{\perp}$.  In fact, even more is true:

\begin{lem}\label{lem_sec_gtau}
The sectional curvatures of $(S^n,\underline g_{\tau})$ satisfy
\begin{subequations}
\label{initialcurvature}
\ba \label{eq:Korder}
K^\top\ge K^\perp_2\ge K_1^\perp\ge L>0,
\ea
and
\ba\label{eq:Ksorder}
 K^\top_s,\,(K_1^\perp)_s,\,(K_2^\perp)_s,\,L_s\ge 0 .
 \ea
\end{subequations}
\end{lem}
\begin{proof}
We compute
\begin{align*}
\frac{ \psi_{\tau}'(r)}{\chi_{\tau}(r)}=\frac{\sin(r)}{\sin^2(r)\sinh^2(-2\tau)-\cosh^2(-2\tau)}, \qquad \frac{\varphi_{\tau}'(r)}{\chi_{\tau}(r)}=\frac{\psi_{\tau}(r)}{\sqrt{\tanh(-2\tau)}}.
\end{align*}
Therefore, 
\begin{align*}
K^\top&=-\frac{\psi_{ss}}{\psi}=-\frac{1}{\psi_{\tau}\chi_{\tau}}\left(\frac{\psi_{\tau}'}{\chi_{\tau}}\right)'\\
&=\frac{1+\sin^2 (r)\tanh^2(-2\tau)}{\left(\sinh(-2\tau)\cosh(-2\tau)\right)\left(1-\sin^2(r)\tanh^2(-2\tau)\right)}.
\end{align*}
In particular, $K^\top$ is positive and monotone increasing in $r$ (and therefore in $s$ as well).

Next, we observe that since $\frac{\varphi_{\tau}'}{\chi_{\tau}\psi_{\tau}}=\frac{(\varphi_{\tau})_s}{\psi_{\tau}}$ is constant in $r$ and $s$, we can conclude that $\frac{(\varphi_{\tau})_{ss}}{\varphi_{\tau}}=\frac{(\varphi_{\tau})_{s}(\psi_{\tau})_{s}}{\varphi_{\tau}\psi_{\tau}}$, so $K_{1}^{\perp}=K_2^{\perp}$ uniformly. 
Observe that 
\begin{align*}
K_2^{\perp}=\frac{\sin(r)}{\cosh^2(-2\tau)\arctanh(\tanh(-2\tau)\sin(r)))\left(1-\sin^2(r)\tanh^2(-2\tau)\right)},
\end{align*}
so that
\begin{align*}
&\sinh(-2\tau)\cosh(-2\tau)\left(1-\sin^2(r)\tanh^2(-2\tau)\right)(K^{\top}-K_2^{\perp})\\
&=1+\sin^2(r)\tanh^2(-2\tau)-\frac{\sin(r)\tanh(-2\tau)}{\arctanh(\tanh(-2\tau)\sin(r))}.
\end{align*}
The non-negativity of the function $1-\frac{x}{\arctanh(x)}$ implies that the right hand side is non-negative. 
Finally, 
\begin{align*}
\arctanh^2(\tanh(-2\tau)\sin(r))L=\tanh(-2\tau)-\frac{\cos^2(r)\tanh(-2\tau)}{1-\sin^2(r)\tanh^2(-2\tau)}
\end{align*}
so that $L$ is strictly positive on $[0,\frac{\pi}{2}]$ (in the sense of limits for $r=0,\frac{\pi}{2}$), and 
\begin{align*}
(1{}&-\sin^2(r)\tanh^2(-2\tau))(K_2^{\perp}-L)\arctanh(\tanh(-2\tau)\sin(r))\\
&=\frac{\sin(r)}{\cosh^2(-2\tau)}-\frac{\tanh(-2\tau)(1-\sin^2(r)\tanh^2(-2\tau))-\cos^2(r)\tanh(-2\tau)}{\arctanh(\tanh(-2\tau)\sin(r))}\\
&=\frac{\sin(r)\arctanh(\tanh(-2\tau)\sin(r))+\cosh(-2\tau)\sinh(-2\tau)\sin^2(r)(\tanh^2(-2\tau)-1)}{\arctanh(\tanh(-2\tau)\sin(r))\cosh^2(-2\tau)},
\end{align*}
which is also non-negative, due to the estimate $\arctanh(x)\ge x$ for $x\ge 0$. 

We have established $K^{\top}\ge K_2^{\perp}= K_1^{\perp}\ge L>0$ and $(K^{\top})_s\ge 0$. The other required derivative estimates follow from \eqref{1dcurv} in Appendix \ref{sec:derivs of curvature}. 
\end{proof}
Since $\underline g_{\tau}$ has positive curvature operator by Lemma \ref{lem_sec_gtau} and \ref{app_curvature}, the work of B\"ohm and Wilking \cite{BohmWilking} implies that the maximal Ricci flow starting at $\underline g_{\tau}$ continues to have positive curvature operator, and develops a singularity in finite time modelled on the round sphere. By shifting time so that the singularity occurs at $t=0$, we thereby obtain a $\tau$-indexed family of Ricci flows 
\[
\big(S^n, g_\tau(t)\big)_{t\in [\alpha_{\tau},0)}, \;\; g_\tau(\alpha_{\tau})= \underline g_{\tau},
\]
with  $\frac{1}{-2(n-1)t} \, g_\tau(t)$ smoothly converging to the round metric as $t\to 0$. 


\subsection{Sturmian properties of the sectional curvatures}\label{sec:curvature relations}
In this section, we establish two remarkable properties of   $O(2)\times O(n-1)$-invariant Ricci flows. 
 
\begin{prop}\label{lem:curvaturepreserved}
The curvature conditions \eqref{eq:Korder}  are preserved under the Ricci flow of $O(2)\times O(n-1)$-invariant metrics on ${S}^n$.
\end{prop}


\begin{proof}
Given any $O(2)\times O(n-1)$-invariant Ricci flow on ${S}^n$, defined on the time interval $[\alpha,0)$,
Hamilton \cite{MR862046} showed that the positivity of the curvature operator is preserved under Ricci flow, so that in particular the inequality $\min\{K^\top,K^\perp_1,K^\perp_2,L\}>0$ is preserved. By \eqref{eq:evolve curvatures} we have
\begin{subequations}\label{eq:evolve_differences}
\ba
\frac{\partial}{\partial t} (K^{\top}-K_2^{\perp})&=\Delta (K^{\top}-K_2^{\perp})-2(n-1)\left(\frac{\varphi_s}{\varphi}\right)^2(K^{\top}-K_2^{\perp})\\
&+2\left(\frac{\psi_s}{\psi}\right)^2 (K_2^{\perp}-K_1^{\perp})+2(K^{\top}-K_2^{\perp})(K^{\top}+K_2^{\perp}-K_1^{\perp}) \nonumber \\
&+2(n-3)K_2^{\perp}(K_1^{\perp}-L), \nonumber
\ea
\ba
\frac{\partial}{\partial t} (K_2^{\perp}-K_1^{\perp})&=\Delta (K_2^{\perp}-K_1^{\perp})-4\left(\frac{\psi_s}{\psi}\right)^2(K_2^{\perp}-K_1^{\perp}) \\
&+2\left(\frac{\varphi_s}{\varphi}\right)^2(K^{\top}-K_2^{\perp}) +2(n-3)\left(\frac{\varphi_s}{\varphi}\right)^2 (K_1^{\perp}-L) \nonumber \\
&+2(K_2^{\perp}-K_1^{\perp})(K_2^{\perp}+K_1^{\perp}-K^{\top}+(n-3)L), \nonumber
\ea
\ba
\frac{\partial}{\partial t}(K_1^{\perp}-L)&=\Delta (K_1^{\perp}-L)+2\left(\frac{\psi_s}{\psi}\right)^2(K_2^{\perp}-K_1^{\perp}) \\
&-2(n-1)\left(\frac{\varphi_s}{\varphi}\right)^2 (K_1^{\perp}-L) \nonumber \\
&+2K_2^{\perp}(K^{\top}-K_2^{\perp})+2(n-3)L(K_1^{\perp}-L).\nonumber
\ea
\end{subequations}
We work in the non-geometric, time-independent coordinate $r\in [0,\frac{\pi}{2}]$. The smoothness conditions on $\chi,\psi,\varphi$ (Proposition \ref{prop:ch1mfsmoothness}) imply that, close to $r=0$, 
\begin{align*}
\chi(r,t)&=\chi(0,t)+r^2 \chi_{\mathsf{waist}}(r,t),\\
 \psi(r,t)&=\psi(0,t)+r^2 \psi_{\mathsf{waist}}(r,t)\\
\varphi(r,t)&=r\chi(0,t)+r^3\varphi_{\mathsf{waist}}(r,t),
\end{align*}
where $\chi(0,t)$ and $\psi(0,t)$ are both positive, and the functions $\chi_{\mathsf{waist}}$, $\psi_{\mathsf{waist}}$, and $\varphi_{\mathsf{waist}}$ are all smooth and even. Similarly, close to $\frac{\pi}{2}$, we have 
\begin{align*}
\chi\left(\tfrac{\pi}{2}-r,t\right)&=\chi\left(\tfrac{\pi}{2},t\right)+r^2\chi_{\mathsf{tip}}(r,t),\\
\psi\left(\tfrac{\pi}{2}-r,t\right)&=r\chi\left(\tfrac{\pi}{2},t\right)+r^3\psi_{\mathsf{tip}}(r,t),\\
\varphi\left(\tfrac{\pi}{2}-r,t\right)&=\varphi\left(\tfrac{\pi}{2},t\right)+r^2 \varphi_{\mathsf{tip}}(r,t),
\end{align*}
where $\chi\left(\frac{\pi}{2},t\right), \varphi\left(\frac{\pi}{2},t\right)$ are positive, and $\chi_{\mathsf{tip}}$, $\psi_{\mathsf{tip}}$, and $\varphi_{\mathsf{tip}}$ are smooth and even with respect to $r=0$. As a result, from \eqref{eqn_seccurv} and $\partial_s = \chi^{-1} \partial_r$ we deduce that the curvatures near the waist satisfy 
\begin{align}\label{cawaist}
\begin{split}
K^{\top}(r,t)&=-\frac{2\psi_{\mathsf{waist}}(0,t)}{\chi(0,t)^2\psi(0,t)}+r^2 K^{\top}_{\mathsf{waist}}(r,t)\\
K_2^{\perp}(r,t)&=-\frac{2\psi_{\mathsf{waist}}(0,t)}{\chi(0,t)^2\psi(0,t)}+r^2 K_{2,{\mathsf{waist}}}^{\perp}(r,t),\\
K_1^{\perp}(r,t)&=-\frac{6\varphi_{\mathsf{waist}}(0,t)}{\chi^3(0,t)}+\frac{2\chi_{\mathsf{waist}}(0,t)}{\chi^3(0,t)}+r^2 K^{\perp}_{1,{\mathsf{waist}}}(r,t),\\
L(r,t)&=-\frac{6\varphi_{\mathsf{waist}}(0,t)}{\chi^3(0,t)}+\frac{2\chi_{\mathsf{waist}}(0,t)}{\chi^3(0,t)}+r^2 L_{\mathsf{waist}}(r,t),
\end{split}
\end{align}
where all of the ${\mathsf{waist}}$ functions are smooth and even in $r$. At the other end, we have 
\begin{align}\label{catip}
\begin{split}
K^{\top}\left(\tfrac{\pi}{2}-r\right)&=-\frac{6\psi_{\mathsf{tip}}(0,t)}{\chi^3\left(\frac{\pi}{2},t\right)}+\frac{2\chi_{\mathsf{tip}}(0,t)}{\chi^3\left(\frac{\pi}{2},t\right)}+r^2 K^{\top}_{\mathsf{tip}}(r,t),\\
K_2^{\perp}\left(\tfrac{\pi}{2}-r\right)&=-\frac{2\varphi_{\mathsf{tip}}(0,t)}{\chi\left(\frac{\pi}{2},t\right)^2\varphi(0,t)}+r^2 K_{2,{\mathsf{tip}}}^{\perp}(r,t),\\
K_1^{\perp}\left(\tfrac{\pi}{2}-r\right)&=-\frac{2\varphi_{\mathsf{tip}}(0,t)}{\chi\left(\frac{\pi}{2},t\right)^2\varphi(0,t)}+r^2 K_{1,{\mathsf{tip}}}^{\perp}(r,t),\\
L\left(\tfrac{\pi}{2}-r\right)&=\frac{1}{\varphi\left(\frac{\pi}{2},t\right)^2}+r^2 L_{\mathsf{tip}}(r,t),
\end{split}
\end{align}
and all of the ${\mathsf{tip}}$ functions are smooth and even around $r=0$. 

Define now a weight function $w:(0,\frac{\pi}{2})\to \mathbb{R}^+$ to be a smooth time-independent function  so that $w(r)=\frac{1}{r^2}$ for $r<\frac{\pi}{6}$, and $w(r)=1$ for $r>\frac{\pi}{3}$. Consider the weighted quantities
\begin{align}\label{weightedquantities}
\begin{split}
X(r,t)&=w(r)(K^{\top}(r,t)-K_2^{\perp}(r,t)),\\
Y(r,t)&=w(\tfrac{\pi}{2}-r)(K_2^{\perp}(r,t)-K_1^{\perp}(r,t)),\\
 Z(r,t)&=w(r)(K_1^{\perp}(r,t)-L(r,t)).
 \end{split}
\end{align}
The functions $X,Y,Z$ of \eqref{weightedquantities} are smooth on $[0,\frac{\pi}{2}]$, and satisfy Neumann conditions at the boundary by \eqref{cawaist} and \eqref{catip}.  We now write evolution equations for these quantities. We have $\partial_s=\frac{\partial_r}{\chi}$, so 
\begin{align*}
\Delta u= u_{ss}+\left(\frac{\psi_s}{\psi}+(n-2)\frac{\varphi_s}{\varphi}\right)u_s=\frac{u_{rr}}{\chi^2}+\left(\frac{\psi_r}{\psi\chi^2}+(n-2)\frac{\varphi_r}{\varphi \chi^2}-\frac{\chi_r}{\chi^3}\right)u_r,
\end{align*} 
implying
\begin{align*}
\Delta \left(\frac{u}{r^2}\right)&=\frac{1}{\chi^2}\left(\frac{u}{r^2}\right)_{rr}+\left(\frac{\psi_r}{\psi\chi^2}+(n-2)\frac{\varphi_r}{\varphi \chi^2}-\frac{\chi_r}{\chi^3}\right)\left(\frac{u}{r}\right)_r\\
&=\frac{1}{\chi^2}\left(\frac{u_{rr}}{r^2}-\frac{4 u_r}{r^3}+\frac{6 u}{r^4}\right)+\left(\frac{\psi_r}{\psi\chi^2}+(n-2)\frac{\varphi_r}{\varphi \chi^2}-\frac{\chi_r}{\chi^3}\right)\left(\frac{u_r}{r^2}-\frac{2u}{r^3}\right)\\
&=\frac{\Delta u}{r^2}-\frac{4u_r}{\chi^2r^3}+u\left(\frac{6}{\chi^2 r^4}-\frac{2}{r^3}\left(\frac{\psi_r}{\psi\chi^2}+(n-2)\frac{\varphi_r}{\varphi \chi^2}-\frac{\chi_r}{\chi^3}\right)\right)\\
&=\frac{\Delta u}{r^2}-\frac{4}{\chi^2r}\left(\frac{u}{r^2}\right)_r-\frac{u}{r^2}\left(\frac{2}{\chi^2 r^2}+\frac{2}{r}\left(\frac{\psi_r}{\psi\chi^2}+(n-2)\frac{\varphi_r}{\varphi \chi^2}-\frac{\chi_r}{\chi^3}\right)\right).
\end{align*}
We use this, together with \eqref{eq:evolve_differences}, to compute 
\begin{align*}
\frac{\partial X}{\partial t}&=\frac{X_{rr}}{\chi^2}+D_1 X_r+E_{11} X+E_{12}Y+E_{13}Z,\\
\frac{\partial Y}{\partial t}&=\frac{Y_{rr}}{\chi^2}+D_2 Y_r+E_{21}X+E_{22}Y+E_{23} Z,\\
\frac{\partial Z}{\partial t}&=\frac{Z_{rr}}{\chi^2}+D_3 Z_{r}+E_{31} X+E_{32}Y+E_{33}Z,
\end{align*}
for $r\in [0,\frac{\pi}{2}]$. 
For $r\in (0,\frac{\pi}{6})$, we get:
\begin{itemize}
\item $D_1=D_3=D_2+\frac{4}{\chi^2 r}=\frac{4}{\chi^2r}+\left(\frac{\psi_r}{\psi\chi^2}+(n-2)\frac{\varphi_r}{\varphi \chi^2}-\frac{\chi_r}{\chi^3}\right)$;
\item $E_{11}=\frac{2}{\chi^2 r^2}+\frac{2}{r}\left(\frac{\psi_r}{\psi\chi^2}+(n-2)\frac{\varphi_r}{\varphi \chi^2}-\frac{\chi_r}{\chi^3}\right)-2(n-1)\left(\frac{\varphi_r}{\chi \varphi}\right)^2+2(K^{\top}+K_2^{\perp}-K_1^{\perp})$;
\item $E_{12}=\frac{2}{r^2}\left(\frac{\psi_r}{\chi\psi}\right)^2$;
\item $E_{13}=2(n-3)K_2^{\perp}$;
\item $E_{21}=2r^2 \left(\frac{\varphi_r}{\chi\varphi}\right)^2$;
\item $E_{22}=-4\left(\frac{\psi_r}{\chi\psi}\right)^2+2(K_2^{\perp}+K_1^{\perp}-K^{\top}+(n-3)L)$;
\item $E_{23}=2(n-3)r^2\left(\frac{\varphi_r}{\chi\varphi}\right)^2$;
\item $E_{31}=2K_2^{\perp}$;
\item $E_{32}=\frac{2}{r^2}\left(\frac{\psi_r}{\chi\psi}\right)^2$;
\item $E_{33}=\frac{2}{\chi^2r^2}+\frac{2}{r}\left(\frac{\psi_r}{\psi\chi^2}+(n-2)\frac{\varphi_r}{\varphi \chi^2}-\frac{\chi_r}{\chi^3}\right)-2(n-1)\left(\frac{\varphi_r}{\chi\varphi}\right)^2+2(n-3)L$. 
\end{itemize}
The key information we need from this list is that for any cutoff time $\sigma\in (\alpha,0)$, the coefficients satisfy the following for $r\in [0,\frac{\pi}{6}]$ and $t\in [\alpha,\sigma]$:
\begin{itemize}
\item $D_1-\frac{n+2}{\chi^2 r}=D_3-\frac{n+2}{\chi^2 r} = D_2-\frac{n-2}{\chi^2 r}$ is a smooth and odd function around $r=0$;
\item $E_{ij}$ are bounded above for all $i,j = 1, \ldots, 3$, and are non-negative when $i\neq j$.
\end{itemize}
The computation for the tip is similar, except we first perform the change $r\mapsto r-\frac{\pi}{2}$ to simplify the outputs: 
\begin{itemize}
\item $D_1=D_3=D_2-\frac{4}{\chi^2r}=\frac{\psi_r}{\psi\chi^2}+(n-2)\frac{\varphi_r}{\varphi \chi^2}-\frac{\chi_r}{\chi^3}$;
\item $E_{11}=-2(n-1)\left(\frac{\varphi_r}{\chi \varphi}\right)^2+2(K^{\top}+K_2^{\perp}-K_1^{\perp})$;
\item $E_{12}=2r^2\left(\frac{\psi_r}{\chi\psi}\right)^2$;
\item $E_{13}=2(n-3)K_2^{\perp}$;
\item $E_{21}=\frac{2}{r^2}\left(\frac{\varphi_r}{\chi \varphi}\right)^2$;
\item $E_{22}=\frac{2}{\chi^2 r^2}+\frac{2}{r}\left(\frac{\psi_r}{\psi\chi^2}+(n-2)\frac{\varphi_r}{\varphi \chi^2}-\frac{\chi_r}{\chi^3}\right)-4\left(\frac{\psi_r}{\psi \chi}\right)^2+2(K_2^{\perp}+K_1^{\perp}-K^{\top}+(n-3)L)$;
\item $E_{23}=\frac{2}{r^2}(n-3)\left(\frac{\varphi_r}{\chi\varphi}\right)^2$;
\item $E_{31}=2K_2^{\perp}$;
\item $E_{32}=2r^2 \left(\frac{\psi_r}{\chi\psi}\right)^2$;
\item $E_{33}=-2(n-1)\left(\frac{\varphi_r}{\chi\varphi}\right)^2+2(n-3)L$. 
\end{itemize}
As for the waist region, here the coefficients $E_{ij}$ are all bounded above, non-negative if $i\neq j$, and $D_1 - \tfrac{1}{\chi^2 r} = D_3 - \tfrac{1}{\chi^2 r} = D_2 - \tfrac{5}{\chi^2 r}$ is smooth and odd around $r=0$ (corresponding to the tip  due to our change of variables).

Now, for each $\varepsilon>0$ and $\sigma<0$, define 
\begin{align*}
u_{\varepsilon,\sigma}=\min\{X,Y,Z\}+\varepsilon e^{(C_{}+1)(t-\alpha)}, \qquad 
C := 3 \max_{[0,\pi/6]\times [\alpha,\sigma]}\{E_{ij}(r,t) : 1\leq i,j \leq 3\}.
\end{align*}
We claim that $u_{\varepsilon,\sigma}(r,t)\ge 0$ holds for all  $(r,t) \in [0,\pi/2]\times [\alpha,\sigma]$, provided that $\min\{X,Y,Z\} \geq 0$ holds at $t= \alpha$. Since $\varepsilon >0$ is arbitrary, the proposition will follow from this. The assumption implies that $u_{\varepsilon, \sigma}(r,\alpha) > 0$ for all $r$. Arguing by contradiction, we let $t_0>\alpha$ be the first time where $u_{\varepsilon,\sigma}(x_0,t_0)=0$ for some $x_0\in [0,\frac{\pi}{2}]$. At least one of the following  holds: 
\begin{align*}
u(x_0,t_0)&=X(x_0,t_0)+\varepsilon e^{(C+1)(t_0-\alpha)}=0 \ \text{and} \ \min\{Y(x_0,t_0),Z(x_0,t_0)\}\ge X(x_0,t_0);\\
u(x_0,t_0)&=Y(x_0,t_0)+\varepsilon e^{(C+1)(t_0-\alpha)}=0 \ \text{and} \ \min\{X(x_0,t_0),Z(x_0,t_0)\}\ge Y(x_0,t_0); \ \text{or} \\
u(x_0,t_0)&=Z(x_0,t_0)+\varepsilon e^{(C+1)(t_0-\alpha)}=0 \ \text{and} \ \min\{X(x_0,t_0),Y(x_0,t_0)\}\ge Z(x_0,t_0).
\end{align*}
Suppose we are in the first case, so that $\tilde{X}(x,t):=X(x,t)+\varepsilon e^{(C+1)(t-\alpha)}\ge 0$ for all $t\le t_0$ and $x\in [0,\frac{\pi}{2}]$, but $\tilde{X}(x_0,t_0)=0$. Then, since $X$ is smooth on $[0,\frac{\pi}{2}]$ with Neumann boundary conditions, we have 
\begin{align}\label{mpoffense}
\tilde{X}_t(x_0,t_0)\le 0, \qquad \tilde{X}_r(x_0,t_0)=0, \qquad \tilde{X}_{rr}(x_0,t_0)\ge 0.
\end{align} Consider now the quantities 
\begin{align*}
    Q(r,t)& =X_t(r,t)-\frac{X_{rr}(r,t)}{\chi^2(r,t)} -D_1(r,t) X_r(r,t), \\
    \tilde{Q}(r,t)&=\tilde{X}_t(r,t)-\frac{\tilde{X}_{rr}(r,t)}{\chi^2(r,t)}-D_1(r,t) \tilde{X}_r(r,t). 
\end{align*}
The fact that $X$ is smooth on the boundary with Neumann conditions,  combined with the $D_1$ estimates at said boundary, imply that $Q$ and $\tilde{Q}$ are continuously extendable to the boundary. We furthermore claim that $\tilde{Q}(x_0,t_0)\le 0$. Indeed, if $x_0\in (0,\frac{\pi}{2})$, then $D_1$ is smooth at $x_0$ and the claim follows immediately from  \eqref{mpoffense}. On the other hand, if $x_0 = 0$,  \eqref{mpoffense} still holds, but we must study the sign of the first order term $D_1 \tilde X$. The $D_1$ expression implies that there is a function $ D_{1,\mathsf{waist}}$ which is smooth at $r=0$, and such that 
\[
    D_1 = D_{1,\mathsf{waist}} + \tfrac{n+2}{\chi^2 r}
\]
for $r\in (0, \pi/6)$. L'H\^opital's rule and \eqref{mpoffense} then yield
\[
    \lim_{r\to 0^+} D_1 \tilde X_r = \lim_{r\to 0^+} D_{1,\mathsf{waist}} X_r + \tfrac{n+2}{\chi^2} \tfrac{X_r}{r} = \tfrac{n+2}{\chi^2(0,t_0)} \, \tilde X_{rr}(0,t_0) \geq 0.
\]
The case $x_0 = \pi/2$ is analogous.

Therefore, 
\begin{align*}
0&\ge \tilde{Q}(x_0,t_0)\\
&=(C+1)\varepsilon e^{(C+1)(t_0-\alpha)}+Q(x_0,t_0)\\
&=(C+1)\varepsilon e^{(C+1)(t_0-\alpha)}+\lim_{r\to x_0}\left(\frac{\partial X}{\partial t}(r,t_0)-X_{rr}(r,t_0)-D_0(r,t_0) X_r(r,t_0)\right)\\
&= (C+1)\varepsilon e^{(C+1)(t_0-\alpha)}+D_1(x_0,t_0)X(x_0,t_0)+D_2(x_0,t_0)Y(x_0,t_0)+D_3(x_0,t_0) Z(x_0,t_0)\\
&\ge (C+1)\varepsilon e^{(C+1)(t_0-\alpha)}+(D_1(x_0,t_0)+D_2(x_0,t_0)+D_3(x_0,t_0))X(x_0,t_0)\\
&= (C+1)\varepsilon e^{(C+1)(t_0-\alpha)}-(D_1(x_0,t_0)+D_2(x_0,t_0)+D_3(x_0,t_0))\varepsilon e^{(C+1)(t_0-\alpha)}>0,
\end{align*}
contradiction. The computations for the second and third cases are very similar.
\end{proof}

In fact, it turns out that all of \eqref{initialcurvature} is preserved by the Ricci flow.
\begin{prop}\label{prop:curvaturepreserved}
The curvature conditions \eqref{eq:Ksorder} are preserved under Ricci flow of $O(2)\times O(n-1)$-invariant metrics on ${S}^n$.
\end{prop}
\begin{proof}
Choose any $O(2)\times O(n-1)$-invariant Ricci flow on ${S}^n$, defined on the time interval $[\alpha,0)$. By Lemma \ref{lem:curvaturepreserved}, we know that 
\begin{align}\label{curvatureordering}
    K^{\top}\ge K_2^{\perp}\ge K_1^{\perp}\ge L>0
\end{align} is preserved, so it suffices to verify the curvature monotonicity conditions 
\begin{align}\label{curvaturemonotonicity}
    K^{\top}_s,(K_1^{\perp})_s,(K_2^{\perp})_s,L_s\ge 0
\end{align} are preserved, under the assumption that our Ricci flow satisfies \eqref{curvatureordering}. Observe that \eqref{1dcurv} implies immediately that $(K_2^{\perp})_s\ge 0$ and $L_s\ge 0$. The other inequalities then follow from the maximum principle, much as for the curvature differences, applied to \eqref{eq:evolve curvature gradients}. Indeed, we can write these evolution equations as 
\begin{align*}
    (\partial_t-\Delta)K_s^{\top}=A_1K_s^{\top}+A_2(K_1^{\perp})_s+A_3\\
    (\partial_t-\Delta)(K_1^{\perp})_s=B_1K_s^{\top}+B_2(K_1^{\perp})_s+B_3\,,
\end{align*}
where $A_3$ and $B_3$ are non-negative, $A_1$ and $B_2$ are bounded from above, and $A_2$ and $B_1$ are non-negative and bounded from above. Consider, for any $\varepsilon>0$ and $\sigma\in(\alpha,0)$, the function
\begin{align*}
    u_{\varepsilon,\sigma}(s,t)=\min\{K_s^{\top}(s,t),(K_1^{\perp})_s(s,t)\}+\varepsilon e^{(C+1)(t-\alpha)}\,,
\end{align*}  
where $C_{\sigma}$ is twice the supremum of $\max\{A_1,A_2,B_1,B_2\}$ on the time interval $[\alpha,\sigma]$. We claim that $u_{\varepsilon,\sigma}$ is non-negative in $[\alpha,\sigma]$. Indeed, since the sectional curvatures are smooth functions satisfying Neumann conditions at the boundary, it follows that, $u_{\varepsilon,\sigma}>0$ at the singular orbits, so if this quantity even became zero, say at a point $(t_0,s_0)$, then $t_0>\alpha$ (since $K_s^{\top}$ and $(K_1^{\perp})_s$ are both non-negative initially), and for $s_0$ corresponding to a principal orbit. Then if $ (K_1^{\perp})_s(s_0,t_0)\ge K_s^{\top}(s_0,t_0)=-\varepsilon e^{(C+1)(t-\alpha)}$, then 
\begin{align*}
    0&\ge (\partial_t-\Delta)\left(\varepsilon e^{(C+1)(t-\alpha)}+K_s^{\top}\right)\vert_{(s_0,t_0)}\\
    &= (C+1)\varepsilon e^{(C+1)(t_0-\alpha)}+A_1(s_0,t_0)K_s^{\top}(s_0,t_0)+A_2(s_0,t_0)(K_1^{\perp})_s(s_0,t_0)+A_3(s_0,t_0)\\
    &\ge (C+1)\varepsilon e^{(C+1)(t_0-\alpha)}+(A_1(s_0,t_0)+A_2(s_0,t_0))K_s^{\top}(s_0,t_0)\\
    &= (C+1-A_1(s_0,t_0)-A_2(s_0,t_0))\varepsilon e^{(C+1)(t_0-\alpha)}>0,
\end{align*}
which is absurd. The case that $K_s^{\top}(s_0,t_0)\ge (K_1^{\perp})_s(s_0,t_0)=-\varepsilon e^{(C+1)(t-\alpha)}$ is treated identically. 
\end{proof}


Thus, our Ricci flows $(S^n,g_\tau(t))_{t\in [\alpha_{\tau},0)}$ satisfy \eqref{initialcurvature} at all times.

\subsection{Crude estimates for key geometric quantities}

Define, along the Ricci flows $\big( S^n, g_\tau(t)\big)_{t\in [\alpha_{\tau},0)}$,
\begin{align}\label{importantquantities}
\begin{aligned}
& \ell_\tau(t) : = \int_{0}^{\frac{\pi}{2}}\chi_\tau(r,t)dr, 
&h_\tau(t) := \psi_\tau(0,t) = \max_{r\in [0,\pi/2]} \psi_\tau(r,t),\\
& A_\tau(t) 
:= 4\pi \int_0^{\ell_\tau(t)}\!\!\!\psi(s,t) ds,
  &d_\tau(t) := \varphi_\tau \big({\pi}/{2},t \big) = \max_{r\in [0,\pi/2]} \varphi_\tau(r, t).
 \end{aligned}
\end{align}
Note that $A_\tau(t)= \area(\Sigma^2, g_\tau^\top(t))$ is the area of the totally geodesic $2$-sphere $\Sigma \subset {S}^{n}$ with respect to the metric $g^\top_\tau(t)$ induced by $g_\tau(t)$, $h_\tau(t)$  is the ``radius'' of the ${S}^1$ singular orbit, $d_\tau(t)$ is the ``radius'' of the ${S}^{n-2}$ singular orbit, and $\ell_\tau(t)$ is the length of a minimising geodesic connecting these two singular orbits. 

We shall derive a series of crude estimates for these quantities.
\begin{lem}\label{areaexistencetime}
Along the Ricci flows $\big(S^n, g_\tau(t)\big)_{t\in [\alpha_{\tau},0)}$ we have 
\begin{align*}
-8\pi t\le A_\tau(t)\le -8\pi(n-1) t  \quad \ \text{and} \quad \frac{-\tau}{n-1}\le -\alpha_{\tau}\le-\tau. 
\end{align*}
\end{lem}
\begin{proof}
We fix $\tau$ and drop the $\tau$ subscripts throughout the proof. 
We directly compute, using the Ricci flow equation,
\begin{align*}
A'(t)&=  \frac12 \int_{\Sigma^2} \tr_{g_\tau^\top}\left(\partial_t \, g_\tau^\top  \right) d \mu_{g_\tau^\top} \\
&=-\int_{{\Sigma}^2}  (2K^{\top}+{(n-2)}(K_1^{\perp}+K_2^{\perp}) )d \mu_{g_\tau^\top}\,.
\end{align*}
Owing to Theorem \ref{prop:curvaturepreserved}, equation \eqref{initialcurvature} and the Gauss--Bonnet theorem, we obtain
\begin{align*}
-8\pi (n-1)\le A'(t)\le -8\pi.
\end{align*}
Integrating this estimate and using $\lim_{t\to 0}A(t)=0$ gives the inequalities for $A(t)$. The estimate for $\alpha_{\tau}$ then follows immediately from $A(\alpha_{\tau})=-8\pi \tau$ \eqref{eqn_area_KR}.
\end{proof}


\begin{lem}\label{hlbound}
Along the Ricci flows $\big( S^n, g_\tau(t)\big)_{t\in [\alpha_{\tau},-(n-1))}$ we have 
\begin{align*}
1  - \frac{n-1}{-t} 
\le h_\tau(t)\le 1 \quad \ \text{and} \quad  - 2 \, t \le \ell_\tau(t)\le  \frac{-4(n-1)t}{1- \frac{n-1}{-t}}.
\end{align*}
\end{lem}

\begin{proof}
We fix $\tau$ and drop $\tau$ subscripts throughout the proof.  Since $K^\top > 0$,  $\psi(s,t)$ is a concave function of $s$, with $\psi_s(0,t) = 0$ by Proposition \ref{prop:ch1mfsmoothness}. Thus, for  $s\in (0,\ell(t))$ we have
\begin{equation}\label{eqn_estimate_psi_h}
	h(t) = \psi(0,t) \geq \psi(s,t) \geq h(t) (1 - s/\ell(t)).
\end{equation}
Curvature positivity also implies that $\psi(s,t)$ is non-increasing in $t$, see \eqref{eq:doubly warped Ricci flow psi}. Hence, $h(t) \leq h(\alpha_\tau) = 1$. 

The formula for the area in \eqref{importantquantities} together with \eqref{eqn_estimate_psi_h} immediately yield 
\begin{align}\label{halestimate}
2 \pi h(t)\ell(t)\le A(t)\le 4 \pi  h(t)\ell(t).
\end{align}
Using $h(t) \leq 1$ and Lemma \ref{areaexistencetime} we get $\ell(t) \geq -2t$. 


Noticing that $2 \, \ell(t)$ is the length of a minimising geodesic in $(\Sigma^2, g_\tau^\top(t))$, the counterpositive of Myer's theorem and  $(K^{\top})_s\ge 0$ (Theorem \ref{prop:curvaturepreserved})  imply
\begin{align}\label{Myer}
K^{\top}(0)=\min_{{\Sigma}^2}K^{\top}\le \frac{\pi^2}{4\ell(t)^2}. 
\end{align}
By using l'H\^opital's rule at $r=0$ in \eqref{eq:doubly warped Ricci flow psi} we then obtain  
\ba
	h'(t) = -(n-1)h(t)K^\top\big|_{r=0}
	\ge{}-\frac{\pi^2(n-1)}{4}\frac{h(t)}{\ell(t)^2}
	\ge -\frac{ \pi^2 (n-1)}{16 t^2} \geq - \frac{n-1}{t^2}.
\ea
Integrating gives our lower bound for $h(t)$. 

Finally, \eqref{halestimate}, Lemma \ref{areaexistencetime} and the lower bound for $h(t)$ give
\bann
	\ell(t) \leq \frac{A(t)}{2\pi h(t)} \leq  \frac{-4(n-1)t}{1-  (n-1) \left( \alpha_\tau^{-1}  - t^{-1} \right)}.
\eann
(Notice that, since $t< -(n-1)$, the denominator in the right-hand-side is positive.)
\end{proof}

%


\begin{lem}\label{dbound}
Along the Ricci flows $\big( S^n,g_\tau(t)\big)_{t\in [\alpha_{\tau},0)}$ we have 
\begin{align*}
d_\tau\ge
\delta\log(-t)-C\;\;\text{for}\;\; t\le -2(n-1)\,,
\end{align*}
where $\delta\doteqdot \frac{1}{4(n-1)}$ and $C\doteqdot \frac{\log(2(n-1))}{4(n-1)}$.
\end{lem}
\begin{proof}
We again fix $\tau$ and drop the $\tau$ subscripts throughout the proof. 
%
Estimating $\varphi\le d$ and $K^\perp_1\le K^\perp_1|_{r=\frac{\pi}{2}}$, we find that
\[
1=-\varphi_s|_{s=\ell(t)}=\int_0^{\ell(t)}(-\varphi_{ss})\,ds=\int_0^{\ell(t)}\varphi K^\perp_1\,ds\le d\ell K^\perp_1|_{r=\frac{\pi}{2}}\,.
\]
Since $K^\perp_2=K_1^\perp$ at the tip, $s = \ell$, \eqref{eq:doubly warped Ricci flow phi} and l'H\^opital's rule then yield
\begin{align*}
d'  \le -\frac{2}{\ell}\,.
\end{align*}
Applying the estimate for $\ell$ from Lemma \ref{hlbound} and integrating from time $t$ to time $-2(n-1)$ yields the claim.
\end{proof}

Finally, we obtain a uniform upper bound for the scalar curvature, and hence the entire curvature operator, through Hamilton's Harnack inequality \cite{HamiltonHarnackRicci}. 

\begin{lem}\label{scalbound}
There exists a constant $C(n)>0$ such that for any sufficiently large $\tau$, and for each $t\in (\tfrac{\alpha_{\tau}}{10}, - 2(n-1))$, the maximal scalar curvature $\Sc^\tau_{\max}$ of the Ricci flow $(S^n,g_\tau(t))_{t\in[\alpha_\tau,0)}$ is bounded by
\begin{align*}
\Sc_{\max}^\tau(t)\le \frac{Ce^{-Ct}}{t^2}. 
\end{align*}
\end{lem}
\begin{proof}
Once again, we drop the $\tau$ superscript. The Harnack inequality implies that
\[
\Sc_{\max}(t)\le\frac{\sigma-\alpha_{\tau}}{t-\alpha_{\tau}}\mathrm{e}^{\frac{\ell^2(t)}{2(\sigma-t)}}\Sc_{\min}(\sigma)
\]
for $\alpha_\tau < t < \sigma < 0$. Setting $\sigma=\frac{t}{2}$ gives 
\begin{align*}
\Sc_{\max}(t)\le\frac{\frac{t}{2}-\alpha_{\tau}}{t-\alpha_{\tau}}\mathrm{e}^{\frac{\ell^2(t)}{-t}}\Sc_{\min}(\tfrac{t}{2})\le10\mathrm{e}^{\frac{\ell^2(t)}{-t}}\Sc_{\min}(\tfrac{t}{2}),
\end{align*}
provided $t\in (\frac{\alpha_{\tau}}{10},0)$. Since $\ell(\frac{t}{2})\ge -2\pi t$ (Lemma \ref{hlbound}), we can use  \eqref{Myer} to conclude that 
\begin{align*}
\Sc_{\min}(\tfrac{t}{2})\le n(n-1)K_{\min}^{\top}\le \frac{n(n-1)}{16 t^2}.
\end{align*}
Estimating
\[	
	\ell(t) \leq - 8 (n-1) t
\]
for $t<-2(n-1)$ via Lemma \ref{hlbound} yields the claim.
\end{proof}

Note that, for $t\ge -2(n-1)$, we may estimate $\ell(t)\le \ell(-2(n-1))\le16(n-1)^2$, which yields the extremely bad (but uniform in $\tau$) estimate
\begin{equation}\label{eq:Sc bound up to time zero}
\Sc_{\max}^\tau\le \frac{Ce^{\frac{C}{-t}}}{t^2},
\end{equation}
where $C=C(n)$.

\subsection{Taking the limit}

We are now in a position to construct the desired $O(2)\times O(n-1)$-invariant ancient Ricci flow on ${S}^n$.

\begin{proof}[Proof of Theorem \ref{mainexistence}]
Choose an arbitrary time interval $I$ compactly contained in $(-\infty,0)$. Since $\alpha_{\tau}\to -\infty$ as $\tau\to -\infty$, the estimates of Lemma \ref{scalbound} and \eqref{eq:Sc bound up to time zero} yield a uniform-in-$\tau$ bound for the Riemann curvature operator $\Rm$ of the Riemannian metrics $\{g_\tau(t)\}_{t\in I}$. 
Shi's estimates \cite{MR1001277} then imply uniform-in-$\tau$ estimates for all spatial derivatives of $\Rm$ (and hence also of the sectional curvatures) on $I$. 

Define now the functions $A_{\tau},B_{\tau}:\mathbb{R}\times I\to \mathbb{R}$  (for sufficiently large $\tau$)
with 
\begin{align*}
A_{\tau}(s,t)=\begin{cases}
\psi_{\tau}(s,t) &\ \text{if}\ 0\le s\le \ell_\tau(t),\\
-\psi_{\tau}(2\ell_\tau(t)-s,t) &\ \text{if}\ \ell_\tau(t)\le s\le 2\ell_\tau(t),\\
-\psi_{\tau}(s-2\ell_\tau(t)) &\ \text{if}\ 2\ell_\tau(t)\le s\le 3\ell_\tau(t),\\
\psi_{\tau}(4\ell_\tau(t)-s,t) & \ \text{if} \ 3\ell_\tau(t)\le s\le 4\ell_\tau(t),
\end{cases}\\
B_{\tau}(s,t)=\begin{cases}
\varphi_{\tau}(s,t) &\ \text{if}\ 0\le s\le \ell_\tau(t),\\
\varphi_{\tau}(2\ell_\tau(t)-s,t) &\ \text{if}\ \ell_\tau(t)\le s\le 2\ell_\tau(t),\\
-\varphi_{\tau}(s-2\ell_\tau(t)) &\ \text{if}\ 2\ell_\tau(t)\le s\le 3\ell_\tau(t),\\
-\varphi_{\tau}(4\ell_\tau(t)-s,t) & \ \text{if} \ 3\ell_\tau(t)\le s\le 4\ell_\tau(t),
\end{cases}
\end{align*}
and then extended to all of $\mathbb{R}$ by insisting that $A_{\tau}(\cdot,t)$ and $B_{\tau}(\cdot,t)$ are both $4\ell_\tau(t)$-periodic. The functions are smooth because of the parity conditions of Proposition \ref{prop:ch1mfsmoothness}. We claim that all derivatives of $A_{\tau}$ and $B_{\tau}$ are uniformly bounded on $I\times \mathbb{R}$, independently of $\tau$. Indeed, 
 Lemma \ref{hlbound} gives us uniform bounds on $\ell_\tau(t)$, and Lemma \ref{scalbound} gives us uniform bounds on $\frac{A_{ss}}{A},\frac{B_{ss}}{B}$; combining with the conditions $A_{\tau}(\ell_\tau(t),t)=0$, $(A_{\tau})_s(\ell_\tau(t),t)=-1$, $B_{\tau}(0,t)=0$, $(B_{\tau})_s(0,t)=1$ gives uniform bounds on $A_{\tau},B_{\tau}$ and their first spatial derivatives. Uniform bounds on all spatial derivatives then follow from the uniform bounds on all spatial derivatives of curvature. Bounds on all mixed space/time derivatives then follows from the evolution equations \eqref{eq:doubly warped Ricci flow} and the Lie bracket equation \eqref{eq:commutator}. 

The Arzel\`a--Ascoli theorem now implies that, for any sequence of times $\tau_j\to-\infty$ and any $k\in \mathbb{N}$, there exist functions $A_I,B_I:\mathbb{R}\times I\to \mathbb{R}$ to which $A_{\tau_j},B_{\tau_j}$ converge in $C^k$ as $j\to \infty$. A standard diagonal argument can be used to extend these limiting functions to all of the time interval $(-\infty,0)$. The lower bounds for $d_\tau(t)$, $h_\tau(t)$, upper and lower bounds for $\ell_\tau(t)$ and positive curvature ensures that these functions can be used to construct an $O(2)\times O(n-1)$-invariant time-varying metric which evolves by Ricci flow and has the same curvature properties \eqref{initialcurvature}. Note that the time interval $(-\infty,0)$ on which our solution has been constructed is maximal due to the upper bound for the area $A$ in Lemma \ref{areaexistencetime}.
\end{proof}

\section{Asymptotics}

We now examine the backwards asymptotics of our constructed ancient Ricci flow $(S^n,g(t))_{t\in (-\infty,0)}$. First note that since $g(t)$ is $O(2)\times O(n-1)$-invariant, it has the form \eqref{ch1mf} for some functions $\chi,\psi,\varphi$. We can again define the geometric quantities \eqref{importantquantities}, and by Lemmas \ref{areaexistencetime}, \ref{hlbound} and \ref{dbound}, our ancient solution satisfies 
\begin{align}\label{ancientcrude}
\begin{split}
-8\pi t\le A(t)\le -8\pi(n-1)t, \qquad 1-\frac{n-1}{-t}\le h(t)\le 1, \\ -2t\le \ell(t) \le -8(n-1)t
, \qquad \frac{1}{4(n-1)}\log\left(\frac{-t}{2(n-1)}\right)\le d(t)
\end{split}
\end{align} 
for all $t<-2(n-1)$. 
We start with asymptotics near the waist:

\begin{lem}
Let $p_{\sf waist} \in S^n$ be a fixed reference point on the ``waist''. Then, the time-translated pointed Ricci flows $(S^n, g(t-\tau), p_{\sf waist})_{t\in (-\infty,\tau)}$ converge smoothly as $\tau \to \infty$ to the stationary flat solution $S^1(2\pi) \times \mathbb{R}^{n-1}$.
\end{lem}

\begin{proof}
By the trace Harnack inequality for ancient Ricci flows \cite{HamiltonHarnackRicci}, $\partial_t \Sc \geq 0$, thus $\Sc$ is uniformly bounded as $\tau \to \infty$. By curvature positivity this implies that the full curvature tensor remains uniformly bounded. The length of the $S^1$-orbit converges to $2 \pi$, for $h(t) \to 1$ as $t\to -\infty$. Hence, for any sequence $\tau_k \to \infty$ there is a subseequence of time-translated Ricci flows which converge to some ancient Ricci flow. The Riemann curvature operator at the $S^1$ singular orbit vanishes as $t\to -\infty$ because of \eqref{Myer} and the $\ell(t)$ estimate in \eqref{ancientcrude}, so by the strong maximum principle, this limiting Ricci flow must be flat. Since $\lim_{t\to -\infty}\ell(t)=\infty$ and $\lim_{t\to -\infty}h(t)=1$, the resulting Ricci flow must be the flat product metric on $S^1\times \mathbb{R}^{n-1}$, and $S^1$ has lenght $2\pi$. Uniqueness of the resulting Ricci flow implies convergence as $\tau \to \infty$, and not just along subsequences. 
\end{proof}

Next we study the asymptotics in the vicinity of the ``tip'' (i.e.~the $S^{n-2}$ singular orbit). To that end, let $(\mathbb{R}^2, g_{\mathrm{cigar},\lambda})$ denote Hamilton's cigar Ricci flow of scale $\lambda$ (with girth equal to $2\pi \lambda$), see Appendix \ref{otherflows}.

\begin{lem}\label{cigarconvergence}
Let $p_{\sf tip}\in S^n$ be a fixed point in the ``tip'' region. There exists $\lambda\in (0,1]$ so that  the time-translated pointed Ricci flows $(S^n, g(t-\tau), p_{\sf tip})_{t\in (-\infty,\tau)}$ converge smoothly as $\tau \to \infty$ to the product Ricci flow $(\mathbb{R}^2, g_{\mathrm{cigar},\lambda}) \times (\mathbb{R}^{n-2}, g_{\rm{flat}})$, 
\end{lem}
\begin{proof}
The Harnack inequality and curvature positivity again imply that the pointed Ricci flows $(g(t-\tau_k),p_{\sf tip})$ converge subsequentialy to another Ricci flow. The estimate for $d(t)$ in \eqref{ancientcrude} implies that the geodesics in $S^{n-2}$ converge to lines, and so the Ricci flow has the direct product form $g^{\top}+g_{\mathrm{flat}}$ on $M\oplus \mathbb{R}^{n-2}$ for some non-compact two-dimensional manifold $M$. Being rotationally-invariant, the ancient Ricci flow $g^{\top}$ must be a cigar soliton of a certain scale $\lambda$, a priori depending on the sequence. 
The estimate for $h(t)$ in \eqref{ancientcrude} implies that the girth $2\pi \lambda$ of this cigar soliton must be no greater than $2\pi$, so $\lambda\le 1$.  

Now, since $\partial_t \Sc(p_{\sf tip}, t) \geq 0$, the limit $\lim_{t\to -\infty} \Sc(p_{\sf tip},t)$ exists. Therefore the scalar curvature at the tip, and hence also the scale $\lambda$, of our limiting cigar soliton is independent of the sequence $(\tau_k)_k$. The uniqueness of the limiting ancient flow then implies that we obtain full convergence, no just along subsequences. 
\end{proof}

Ultimately, we would like to show that the asymptotic cigar has the right scale $\lambda = 1$, making it compatible with the asymptotic limit in the waist region.

\begin{prop}\label{prop:lambda=1}
The scale of the cigar factor in the asymptotic limit on the tip region (Lemma \ref{cigarconvergence}) is $\lambda=1$.
\end{prop}

To prove this, we examine first how the scale $\lambda$ affects the geometry of the ancient solution. 
\begin{lem}\label{scalecigarlength}
For each $\bar{\lambda}>\lambda$, there exists $t_0 = t_0(\lambda, \bar\lambda)$ such that, for all $t < t_0$, we have the length estimate
\begin{equation}\label{eq:sharpish length estimate}
\ell(t)\ge-2\bar{\lambda}^{-1}t-C.
\end{equation}
where $C:= -2\bar{\lambda}^{-1}t_0-\ell(t_0)$.
\end{lem}
\begin{proof}
In order to estimate the time derivative of $\ell$, we introduce the time-independent parameters $\delta = \delta(\lambda, \bar \lambda) > 0$, $\varepsilon = \varepsilon(\delta, \lambda, \bar \lambda) > 0$, $t_0 = t_0 (\delta, \varepsilon) < 0$. Given $\varepsilon$ and $\delta$, we choose $t_0$ small enough so that 
\[
		\Rc_1(s,t) \geq K_{\mathrm{cigar},\lambda}(\ell-s) - \varepsilon
\]
for all $s \in [\ell-\delta, \ell]$ and all $t < t_0$; this is possible because of the backwards time convergence established in Lemma \ref{cigarconvergence}. We compute: 
\ba
	- \frac{d\ell}{dt} =& \int_0^{\ell} \Rc_1(s,t) ds   \geq \int_{\ell-\delta}^\ell \Rc_1(s,t) ds  \nonumber\\ 
		\geq& \int_{0}^\delta K_{\mathrm{cigar},\lambda}(u) du - \varepsilon \delta  = \frac{2}{\lambda} \tanh(\delta/\lambda) - \varepsilon \delta   \nonumber \\
		\geq& \frac{2}{\lambda} (1 - 2 e^{-2\delta/\lambda}) - \varepsilon \delta  \geq \frac{2}{\bar \lambda}, \nonumber
\ea
where $u$ denotes the arc-length parameter of the cigar geodesic (scale $\lambda$), and in the last step we chose first $\delta$ and then $\varepsilon$ so that
\[
	\frac{4}{\lambda} e^{-2\delta/\lambda} = \frac1{\lambda} - \frac1{\bar\lambda} = \varepsilon \delta .
\]
The lemma follows by integrating the above estimate on $[t,t_0]$.
\end{proof}

On the other hand, we obtain the following upper bound for the area.
\begin{lem}\label{areaupper}
For each small $\varepsilon>0$, there exists  $t_0<0$ and $C(\varepsilon)>0$ such that
\ba
\frac{A(t)}{2}
\le  -4\pi(1+\varepsilon)t+C(\varepsilon)\,  \qquad \hbox{for all } t\leq t_0.\label{eq:sharpish area estimate}
\ea
\end{lem}
\begin{proof}
Observe that
\begin{align*}
\frac{A(t)}{2}
&=\int_t^0- \frac{A'(\tau)}{2}d\tau\\
&=\int_t^0 \int_{\Sigma^2} (K^{\top}+(n-2)K^{\perp} ) \, d\mu_{g^\top} d\tau\\
&=-4\pi t+(n-2)\int_t^0 \int_{\Sigma^2} K^{\perp} d\mu_{g^\top} d\tau. 
\end{align*}

To deal with the second term, first observe that   by \eqref{initialcurvature} we have
\[
		\int_{-1}^0 \int_{\Sigma^2} K^\perp d \mu_{g^\top} d\tau \leq \int_{-1}^0 \int_{\Sigma^2} K^\top d\mu_{g^\top} d\tau = 4\pi.
\]
Regarding the interval $(t,-1)$, we use the estimate
\begin{equation}\label{eq:int Kperp dt estimate}
\int_t^{-1}K^\perp(\tfrac{\pi}{2},\tau)\,d\tau\le O(\log(-t))\;\;\text{as}\;\;t\to-\infty,
\end{equation}
which follows from the estimates 
\[
-\frac{d}{dt}\log d(t)\ge 2K^\perp|_{r=\frac{\pi}{2}}\;\;\text{and}\;\; d\le O(-t)\,.
\]
This is enough to proceed analogously to \cite{BLT1} to obtain \eqref{eq:sharpish area estimate}. To that end, we introduce the coordinate $\rho = \rho(s,t) := \arcsin(-\psi_s(s))\in (0,\pi/2)$
(recall that $\psi_s$ takes values on $(-1,0)$). For a fixed $\rho_0 >0$ we also consider $s_0 = s_0(\rho_0,t)$ defined so that $\rho(s_0, t) = \rho_0$, i.e.
\[
	\sin \rho_0 = -\psi_s(s_0).
\] 
This yields a ``partial Gauss--Bonnet theorem'' on $\Sigma^2_{<s_0} := \{ s < s_0 \} \subset \Sigma^2$:
\begin{equation} \label{eqn:partial_GB}
	\int_{\Sigma_{<s_0}^2} K^\top d \mu_{g^\top} = 2 \int_0^{s_0}  2 \pi \psi K^\top ds = - 4\pi \int_0^{s_0} \psi_{ss} ds =  4 \pi \sin(\rho_0).
\end{equation}
Observe also that, since $\rho_0$ is fixed, $\ell - s_0$ stays uniformly bounded independent of $t$. Indeed, using that $-\psi_s$ is non-decreasing we have
\[
	\sin(\rho_0) (\ell - s_0) \leq \int_{s_0}^\ell -\psi_s(s) ds = \psi(s_0) \leq 1.
\]
This implies the uniform-in-time area estimate
\begin{equation}\label{eqn:uniform area estimate geq s_0}
	\int_{\Sigma^2_{\geq s_0}}  d\mu_{g^\top} = 2 \int_{s_0}^\ell 2\pi \psi ds \leq 4\pi (\ell-s_0) \leq \frac{4\pi}{\sin(\rho_0)}.
\end{equation}

Thus, for each  $t<0$ we split $\Sigma^2 = \Sigma^2_{<s_0} \cup \Sigma^2_{\geq s_0}$ and  use \eqref{initialcurvature}, \eqref{eqn:partial_GB} and \eqref{eqn:uniform area estimate geq s_0} to obtain
\begin{align}\label{eqn:estimate int Kperp}
	\int_{\Sigma^2} K^\perp(t) d \mu_{g^\top} \leq&   \,\,    \int_{\Sigma^2_{<s_0}} K^\top d\mu_{g^\top}  + K^\perp(\tfrac{\pi}{2}, t) \, \int_{\Sigma^2_{\geq s_0}}  d\mu_{g^\top} \\
		\leq& \, \, 4\pi \sin(\rho_0)    +  K^\perp(\tfrac{\pi}{2}, t) \, \frac{4\pi}{\sin(\rho_0)}. \nonumber
\end{align}
Finally, integrating \eqref{eqn:estimate int Kperp} and using \eqref{eq:int Kperp dt estimate} we immediately get \eqref{eq:sharpish area estimate} after choosing $\rho_0>0$ small enough.
\end{proof}

\begin{proof}[Proof of Proposition \ref{prop:lambda=1}]
Suppose, to the contrary, that $\lambda<1$. 
For each $t<0$ and $d_0\in (0,\ell(t))$, curvature positivity implies
\begin{align*}
\psi(s,t)\ge \frac{\left(\ell(t)-d_0-s\right)\psi(0,t)+s\psi(\ell(t)-d_0,t)}{\ell(t)-d_0}
\end{align*}
for all $s\in (0,\ell(t)-d_0)$. 
Therefore, we can estimate the $g^{\top}$ area of $S^2$ with 
\begin{align*}
\frac{A(t)}{2}&=2\pi \int_{0}^{\ell(t)}\psi(s,t)ds\\
&\ge 2\pi \int_{0}^{\ell(t)-d_0}\psi(s,t)ds\\
&\ge \pi (\ell(t)-d_0)\left(\psi(\ell(t)-d_0,t)+\psi(0,t)\right).
\end{align*}
Lemmas \ref{scalecigarlength} and \ref{areaupper} then imply that for each $\varepsilon>0$ and $\bar{\lambda}\in (\lambda,1)$, there is a $t_0,C_1,C_2$ so that for all $t\le t_0$, we have 
\begin{align*}
\left(2\bar{\lambda}^{-1}+ \frac{C_1-d_0}{-t}\right)\left(\psi(\ell(t)-d_0,t)+\psi(0,t)\right)&\le \frac{(\ell(t)-d_0)}{-t}\left(\psi(\ell(t)-d_0,t)+\psi(0,t)\right) \\
&\le \frac{A(t)}{-2\pi t}\\
&\le 4 (1+\varepsilon)+\frac{C_2}{-t}.
\end{align*}
But this is impossible, because $\lim_{d_0\to \infty}\left(\lim_{t\to -\infty}\psi(\ell(t)-d_0,t)+\psi(0,t)\right)=\lambda+1$. 
\end{proof}

Finally, we treat asymptotics of points that have neither bounded distance to the $S^1$ singular orbit, nor to the $S^{n-2}$ singular orbit. 
\begin{lem}
For any sequence of times $\{t_k\}_{k=1}^{\infty}$ with $\lim_{k\to \infty}t_k= -\infty$, and any sequence of points $\{p_k\}_{k=1}^{\infty}$ with unbounded $g(t_k)$ distance to the $S^{n-2}$ and $S^1$-singular orbits, the time translated pointed Ricci flows $(S^n,g(t-t_k),p_k)_{t\in(-\infty,-t_k)}$ converge to the stationary flat cylinder $S^1\times \mathbb{R}^{n-1}$ (of unit radius). 
\end{lem}
\begin{proof}
Consider the arc-length parametrisation of the functions $\psi_k$ and $\varphi_k$, recentered so that $s=0$ corresponds to $p_k$. Since the distance to both singular orbits is increasing without bound, $(\varphi_k)_s\ge 0$, $(\varphi_k)_{ss}\le 0$ and $\varphi_k$ is also increasing without bound at the $S^{n-2}$ singular orbit, we conclude that $\frac{1}{\varphi_k}$ converges to $0$ uniformly on compact subsets (of the $s$ variable).  
Similarly, since $(\psi_k)_s\le 0$, $\psi_k$ is converging to $1$ at the singular orbits, and the girth of the cigar solitons forming near the $S^{n-2}$ orbit is $2\pi$, we conclude that $\psi_k$ must be converging to $1$ uniformly on compact subsets. The convergence of the $\psi_k$ and $\varphi_k$ functions give us the required geometric convergence. 
\end{proof}

\appendix

\section{$O(2)\times O(n-1)$-invariant Riemannian metrics on $S^n$}\label{geometry}

The purpose of this appendix is to collect a number of important, but computationally-intensive observations regarding the geometry of the Riemannian metrics on ${S}^{n}$ which have the form \eqref{ch1mf}.

\subsection{Orthonormal frames for doubly warped product metrics}\label{oframes}
To begin, observe that there is an $O(2)\times O(n-1)$-invariant diffeomorphism from the principal part of $S^{n}$ and $(0,\frac{\pi}{2})\times S^1\times S^{n-2}$, on which the Riemannian metric takes the form 
\begin{align*}
\chi(r)^2 dr^2+\psi(r)^2 d\theta^2+\varphi(r)^2 g_{S^{n-2}},
\end{align*}
where $\theta$ is the usual parameter for ${S}^1$, and $g_{{S}^{n-2}}$ is the round metric of radius $1$ on ${S}^{n-2}$. This doubly warped product structure allows us to compute many of the important geometric terms relevant to this paper. 

We start by defining $\tilde{e}_1=\partial_r$, the vector field evolving the natural co-ordinate on $(0,\frac{\pi}{2})$, $\tilde{e}_2=\partial_{\theta}$, the vector field evolving the natural co-ordinate on ${S}^1$, and a frame of vector fields $\{\tilde{e}_k\}_{k=3}^{n}$ which are orthonormal with respect to the round radius $1$ metric $g_{{S}^{n-2}}$ on the ${S}^{n-2}$ component, locally defined around a point $p\in {S}^{n-2}$. Note that $[\tilde{e}_i,\tilde{e}_j]=0$, unless $i$ and $j$ are both in $\{3,\cdots,n\}$. Then an orthonormal frame of vector fields for $g$ around the point $(r,\theta,p)$ is given by 
\begin{align*}
e_1=\frac{\tilde{e}_1}{\chi(r)}, \qquad e_2=\frac{\tilde{e}_2}{\psi(r)}, \qquad e_i=\frac{\tilde{e}_i}{\varphi(r)} \ \text{for all} \ i=3,\cdots,n.
\end{align*}
Observe that for $i,j\in \{3,\cdots,n\}$, 
\begin{align*}
[e_1,e_2]=-\frac{\psi'}{\chi\psi}e_2, \qquad [e_1,e_j]=-\frac{\varphi'}{\chi\varphi}e_j, \qquad [e_2,e_j]=0, \qquad [e_i,e_j]=\frac{1}{\varphi^2}[\tilde{e}_i,\tilde{e}_j]. 
\end{align*}
We find it convenient to introduce the vector field $\partial_s=\frac{\partial_r}{\chi}$ so that the previous Lie bracket relations become 
\begin{align*}
[e_1,e_2]=-\frac{\psi_s}{\psi}e_2, \qquad [e_1,e_j]=-\frac{\varphi_s}{\varphi}e_j, \qquad [e_2,e_j]=0, \qquad [e_i,e_j]=\frac{1}{\varphi^2}[\tilde{e}_i,\tilde{e}_j]. 
\end{align*}
We also define $\omega_i$ to be the dual one-form associated to $e_i$. 
\subsection{The Levi-Civita Connection}
Use of the Koszul formua
\begin{align*}
g(\nabla_{e_i}e_j,e_k)=g([e_i,e_j],e_k)+g([e_k,e_i],e_j)+g([e_k,e_j],e_i),
\end{align*} allows for the computation of the $(1,1)$-tensor field $\nabla e_i$ for each $i\in \{1,\cdots,n\}$:
\begin{align}\label{LCC}
\begin{split}
\nabla e_1&=\frac{\psi_s}{\psi}\omega_2\otimes e_2+\sum_{j=3}^{n}\frac{\varphi_s}{\varphi}\omega_j\otimes e_j,\\
\nabla e_2&=-\frac{\psi_s}{\psi}\omega_2\otimes e_1, \ \text{and}\\
\nabla e_i&=-\frac{\varphi_s}{\varphi}\omega_i\otimes e_1+\sum_{k,l=3}^{n}\frac{\tilde{\Gamma}_{kil}}{\varphi}\omega_k\otimes e_l \ \text{for} \ i=3,\cdots,n,
\end{split}
\end{align}
where $\tilde{\Gamma}_{ijk}$ are the Christoffel symbols of $(S^{n-2},g_{S^{n-2}})$ in the basis $\{\tilde{e}_i\}_{i=3}^{n}$. 
Similarly, we can compute the covariant derivatives of $\omega_i$ using the rule $\nabla_{X}(\omega_i)(Y)=X(\omega_i(Y))-\omega_i(\nabla_{X}Y)$:
\begin{align*}
\nabla \omega_1&=\frac{\psi_s}{\psi}\omega_2\otimes \omega_2+\sum_{j=3}^{n}\frac{\varphi_s}{\varphi}\omega_j\otimes \omega_j,\\
\nabla \omega_2&=-\frac{\psi_s}{\psi}\omega_2\otimes \omega_1, \ \text{and}\\
\nabla \omega_i&=-\frac{\varphi_s}{\varphi}\omega_i\otimes \omega_i+\sum_{k,l=3}^{n}\frac{\tilde{\Gamma}_{kil}}{\varphi}\omega_k\otimes \omega_l \ \text{for} \ i=3,\cdots,n,
\end{align*}

\subsection{Curvature}\label{app_curvature}
Using the expressions in \eqref{LCC} we compute that for each point $x$ in a neighbourhood of $(r,\theta,p)\in (0,\frac{\pi}{2})\times {S}^1\times {S}^{n-2}$, the Riemann curvature operator $\Rm:T_x {S}^{n}\wedge T_x S^{n}\to T_x {S}^{n}\wedge T_x {S}^{n}$ is diagonal in the basis $\{e_i\wedge e_j\}_{1\le i<j\le n}$, and is given by 
\begin{align}\label{eqn_seccurv}
\begin{split}
K^\top := \Rm_{1212}&=-\frac{\psi_{ss}}{\psi}\\
 K_1^\perp := \Rm_{1i1i}&=-\frac{\varphi_{ss}}{\varphi},\\
 K_2^\perp := \Rm_{2i2i}&=-\frac{\psi_s\varphi_s}{\psi\varphi},\\
 L := \Rm_{ijij}&=\frac{1-\varphi_s^2}{\varphi^2},
\end{split}
\end{align}
for all $i,j\in \{3,\cdots,n\}$ with $i\neq j$. 
Then the Ricci curvature is diagonal in the basis $\{e_i\}_{i=1}^{n}$ and is given by 
\begin{align}\label{ch1ricc}
\begin{split}
\Rc_{11}&=-\frac{\psi_{ss}}{\psi}-(n-2)\frac{\varphi_{ss}}{\varphi},\\
\Rc_{22}&=-\frac{\psi_{ss}}{\psi}-(n-2)-\frac{\psi_s\varphi_s}{\psi\varphi},\\
\Rc_{ii}&=-\frac{\varphi_{ss}}{\varphi}-\frac{\psi_s\varphi_s}{\psi\varphi}+(n-3)\frac{1-\varphi_s^2}{\varphi^2}, \ \text{for each} \ i=3,\cdots,n.
\end{split}
\end{align}

\subsection{Derivatives of curvature}\label{sec:derivs of curvature}

We produce some useful formulas which relate derivatives of the sectional curvatures in terms of the warping functions $\chi,\psi,\varphi$, as well as the sectional curvatures themselves. First, observe that
\begin{align}\label{1dcurv}
\begin{split}
(K_2^{\perp})_s&=\frac{\varphi_s}{\varphi}(K^{\top}-K_2^{\perp})-\frac{\psi_s}{\psi}(K_2^{\perp}-K_1^{\perp})\\
L_s&=\frac{2\varphi_s}{\varphi}(K_1^{\perp}-L). 
\end{split}
\end{align}
The tensor Laplacian $\Delta \Rm$ can be computed through use of the rule for covariant derivatives of a $(2,2)$-tensor field $T$:
\begin{align*}
&(\nabla_{e_k}T)(e_a,e_b,\omega_c,\omega_d)-e_k(T(e_a,e_b,\omega_c,\omega_d))\\
&=-T(\nabla_{e_k}e_a,e_b,\omega_c,\omega_d)-T(e_a,\nabla_{e_k}e_b,\omega_c,\omega_d)-T(e_a,e_b,\nabla_{e_k}\omega_c,\omega_d)-T(e_a,e_b,\omega_c,\nabla_{e_k}\omega_d).
\end{align*}
We conclude that 
 that $\Delta \Rm=\sum_{k=1}^{n}\nabla_{e_k}\nabla_{e_k}\Rm-\nabla_{\nabla_{e_k}e_k}\Rm$ is diagonal, and that 
\begin{align*}
(\Delta \Rm)(e_i,e_j,\omega_i,\omega_j)-{}&\Delta (\Rm(e_i,e_j,\omega_i,\omega_j))\\
&=2\sum_{k=1}^{n}\Rm(\nabla_{e_k}\nabla_{e_k}e_i,e_j,\omega_i,\omega_j)+\Rm(e_i,\nabla_{e_k}\nabla_{e_k}e_j,\omega_i,\omega_j)\\
&+2\sum_{k=1}^{n}\Rm(\nabla_{e_k}e_i,e_j,\nabla_{e_k}\omega_i,\omega_j)+\Rm(e_i,\nabla_{e_k}e_j,\omega_i,\nabla_{e_k}\omega_j)\\
&+4\sum_{k=1}^{n}\Rm(\nabla_{e_k}e_i,\nabla_{e_k}e_j,\omega_i,\omega_j)+\Rm(\nabla_{e_k}e_i,e_j,\omega_i,\nabla_{e_k}\omega_j,,
\end{align*}
with the scalar Laplacian given by $\Delta f=f_{ss}+\left(\frac{\psi_s}{\psi}+(n-2)\frac{\varphi_s}{\varphi}\right)f_s$. 
Note that the terms of the form $e_k(\Rm(\nabla_{e_k}e_i,e_j,\omega_i,\omega_j))$ all vanish for $k=1$ because the frame is parallel along the $\partial_r$ direction. They also vanish for $k\ge 2$ because the term in brackets is constant in the direction of the principal orbits. 

Completing the computation gives:
\begin{align*}
(\Delta \Rm)_{1212}&=\Delta K^{\top}-2(n-2)\left(\frac{\varphi_s}{\varphi}\right)^2 (K^{\top}-K_2^{\perp})\,,\\
(\Delta \Rm)_{1i1i}&=\Delta K_1^{\perp}+2\left(\frac{\psi_s}{\psi}\right)^2(K_2^{\perp}-K_1^{\perp})-2(n-3)\left(\frac{\varphi_s}{\varphi}\right)^2(K_1^{\perp}-L)\,,\\
(\Delta \Rm)_{2i2i}&=\Delta K_2^{\perp}-2\left(\frac{\psi_s}{\psi}\right)^2(K_2^{\perp}-K_1^{\perp})+2\left(\frac{\varphi_s}{\varphi}\right)^2(K^{\top}-K_2^{\perp})\,,\\
(\Delta \Rm)_{ijij}&=\Delta L+4\left(\frac{\varphi_s}{\varphi}\right)^2(K_1^{\perp}-L)\,.
\end{align*}

\subsection{Square and Lie algebra square of curvature}

Take a principal point $p\in S^n$ and the orthonormal frame $\{e_i\}_{i=1}^{n}$ for $V=T_p S^n$ described in \ref{oframes}. This orthonormal frame allows us to identify $V\wedge V$ with $\mathfrak{so}(n)$ via $e_i\wedge e_j\mapsto e_{ij}-e_{ji}$. The set $\{e_i\wedge e_j\}_{1\le i<j\le n}$ then forms an orthonormal basis for $\mathfrak{so}(n)$ equipped with the inner product $\langle A,B\rangle_{\mathfrak{so}(n)}=-\frac{1}{2}\tr(AB)$. This identification allows us to interpret the Riemann curvature operator $\Rm$ at the point $p$ as a self-adjoint linear operator on $\mathfrak{so}(n)$, and the formula for its Lie algebra square is
\begin{align*}
    \langle \Rm^{\sharp}(h),h\rangle =\frac{1}{2}\sum_{\alpha,\beta=1}^N\langle [\Rm(b_{\alpha}),\Rm(b_{\beta})],h\rangle \cdot \langle [b_{\alpha},b_{\beta}],h\rangle, 
\end{align*}
where $\{b_{\alpha}\}_{\alpha=1}^{N}$ is an orthonormal basis for $\mathfrak{so}(n)$. 
Since $\Rm$ is diagonal in the $\{e_i\wedge e_j\}_{1\le i<j\le n}$ basis, we find immediately that $\Rm^{\sharp}$ is also diagonal, and we compute 
\begin{align*}
(\Rm^2+\Rm^{\sharp})_{1212}&=(K^{\top})^2+(n-2)K_1^{\perp}K_2^{\perp}\,,\\
(\Rm^2+\Rm^{\sharp})_{1212}&=(K^{\perp}_1)^2+K^{\top}K_2^{\perp}+(n-3)K_1^{\perp}L\,,\\
(\Rm^2+\Rm^{\sharp})_{2i2i}&=(K^{\perp}_2)^2+K^{\top}K_1^{\perp}+(n-3)K_2^{\perp}L\,,\\
(\Rm^2+\Rm^{\sharp})_{ijij}&=(K^{\perp}_1)^2+(K_2^{\perp})^2+(n-3)L^2\,,
\end{align*}
for all $i,j\in \{3,\cdots,n\}$ with $i\neq j$.

\section{Cohomogeneity one ancient Ricci flows in low dimensions}\label{otherflows} 

The purpose of this appendix is to give detailed accounts of a number of low-dimensional ancient Ricci flows that are highly relevant to our construction. 

\subsection{Hamilton's cigar}
The cigar solution of the Ricci flow is a time-varying $O(2)$-invariant metric $g_{\mathrm{cigar}}(t)$ on $\mathbb{R}^2$ given in polar coordinates by 
\begin{align*}
g_{\mathrm{cigar}}(t)=\frac{e^{2t}(\cosh^2(r)dr\otimes dr+\sinh^2(r)d\theta\otimes d\theta)}{1+e^{2t}\sinh^2(r)}.
\end{align*}
By introducing the arc-length parameter $s(r,t)=\arcsinh(\sinh(r)e^t)$, we can construct a time-varying $O(2)$-invariant diffeomorphism of $\mathbb{R}^2$ so that the pull-back metric is 
\begin{align*}
g_{\mathrm{cigar}}=ds\otimes ds+\tanh^2(s)d\theta\otimes d\theta;
\end{align*}
since there is no explicit $t$-dependence, we can conlude that $g_{\mathrm{cigar}}$ is a steady Ricci soliton. 
The scalar curvature is $4\sech^2(s)$; this is  $4$ at the tip, and vanishes in the limit as $s\to \infty$. In particular, the cigar soliton is not $\kappa$-noncollapsed on all scales, since the limit of the length of the $S^1$ principal orbits under the $O(2)$ action is $2\pi$. We shall refer to $g_{\mathrm{cigar}}$ as the cigar soliton \emph{of scale $1$}. 
The cigar soliton \emph{of scale $\lambda$} is given in arclength parametrisation by 
\begin{align*}
g_{\mathrm{cigar},\lambda}=ds\otimes ds+\lambda^2\tanh^2(\lambda^{-1}s)d\theta\otimes d\theta\,.
\end{align*}

\subsection{The ancient sausage solution}\label{app_KR}
The ancient sausage solution is an $O(2)$-invariant ancient Ricci flow on $S^2$ given by 
\begin{align*}
g_{\mathrm{sausage}}(t)=\chi(r,t)^2 (dr\otimes dr+\cos^2(r)d\theta\otimes d\theta),
\end{align*}
where $\chi(r,t)=\sqrt{\frac{\tanh(-2t)}{1-\sin^2(r)\tanh^2(-2t)}}$. Here, $r=0$ corresponds to the equator, and $r=\pm \frac{\pi}{2}$ corresponds to the singular orbits of the $O(2)$ action. 

Observe that this Ricci flow has positive sectional curvature (see the computation in the proof of Lemma \ref{initialcurvature}), so near the singularity at $t=0$, the normalised Ricci flow converges smoothly to the round sphere. Also, since by Gauss--Bonnet the area of a Ricci flow on $S^2$ decays  linearly at a speed of $-8\pi$, we have that 
\begin{equation}\label{eqn_area_KR}
	\area(S^2,g_{\mathrm{sausage}}(t)) = -8\pi t.
\end{equation} 
For the backwards asymptotics, we introduce the co-ordinate $R=\frac{\cosh(-2t)}{\tan(r)}$, so that the metric becomes 
\begin{align*}
g_{\mathrm{sausage}}(t)=\frac{\tanh(-2t)}{1+R^2}\left(\frac{\cosh(-2t)^2}{\cosh(-2t)^2+R^2}dR\otimes dR+R^2 d\theta\otimes d\theta\right);
\end{align*}
this converges to the metric $\frac{dR\otimes dR+R^2 d\theta\otimes d\theta}{1+R^2}$, which we recognise as the cigar soliton with $R=\sinh(\rho)$. The supremum length of the $S^1$ fibers is $2\pi$. 

Since the ancient sausage solution has positive curvature, and the length of the $S^1$ orbit at the equator is converging to $2\pi$, the other asymptotics must be converging to the cylinder $S^1\times \mathbb{R}$ of waist $2\pi$.

\subsection{The $O(2)\times O(2)$-invariant hypersausage solution on $S^3$}

In order to put our solutions into some context, we recall that Fateev's ancient ``hypersausage'' solution is an $O(2)\times O(2)$-invariant Ricci flow on $S^3$ which is given, in the principal domain $(0,\frac{\pi}{2})\times S^1\times S^1$, by
\[
g_{\mathrm{hypersausage}}=\chi^2dr^2+\psi^2d\theta^2+\varphi^2d\omega^2\,,
\]
where
\bann
\chi^2(r,t):={}&\frac{\cosh(-2t)\sinh(-2t)}{2(\cos^2r+\sin^2r\cosh(-2t))(\sin^2r+\cos^2r\cosh(-2t))}\\
\psi^2(r,t):={}&\frac{\cos^2r\sinh(-2t)}{2(\sin^2r+\cos^2r\cosh(-2t))}\\
\varphi^2(r,t):={}&\frac{\sin^2r\sinh(-2t)}{2(\cos^2r+\sin^2r\cosh(-2t))}\,.
\eann

Note that, as $t\to -\infty$, both $\psi(0,t)$ and $\varphi(\frac{\pi}{2},0)$ are bounded. Since only one of these quantities is bounded for the ancient Ricci flow on $S^3$ of Theorem \ref{mainexistence}, the two solutions are not isometric. 

\begin{rem}This solution corresponds to ``$k=0$'' in the family of ``twisted hypersausage'' solutions on $S^3$ constructed by Bakas--Kong--Ni \cite{BakasKongNi} (following Fateev \cite{Fateev1,Fateev2}). But the other metrics in the family 
are not $O(2)\times O(2)$-invariant.
\end{rem}

\bibliographystyle{plain}
\bibliography{bibliography}

\begin{thebibliography}{10}

\bibitem{LynchAbrego}
Andoni~Royo Abrego and Stephen Lynch.
\newblock Ancient solutions of {R}icci flow with type-{I} curvature growth.
\newblock {\em The Journal of Geometric Analysis}, 34(119), 2024.

\bibitem{Almgren}
Frederick~Justin Almgren, Jr.
\newblock The homotopy groups of the integral cycle groups.
\newblock {\em Topology}, 1:257--299, 1962.

\bibitem{AndrewsHopper}
Ben Andrews and Christopher Hopper.
\newblock {\em The {R}icci flow in {R}iemannian geometry}, volume 2011 of {\em Lecture Notes in Mathematics}.
\newblock Springer, Heidelberg, 2011.
\newblock A complete proof of the differentiable 1/4-pinching sphere theorem.

\bibitem{MR4400906}
Sigurd Angenent, Simon Brendle, Panagiota Daskalopoulos, and Natasa \v{S}e\v{s}um.
\newblock Unique asymptotics of compact ancient solutions to three-dimensional {R}icci flow.
\newblock {\em Comm. Pure Appl. Math.}, 75(5):1032--1073, 2022.

\bibitem{BakasKongNi}
Ioannis Bakas, Shengli Kong, and Lei Ni.
\newblock Ancient solutions of {R}icci flow on spheres and generalized {H}opf fibrations.
\newblock {\em J. Reine Angew. Math.}, 663:209--248, 2012.

\bibitem{Birkhoff}
George~D. Birkhoff.
\newblock Dynamical systems with two degrees of freedom.
\newblock {\em Trans. Amer. Math. Soc.}, 18(2):199--300, 1917.

\bibitem{MR3984075}
Christoph B\"ohm, Ramiro Lafuente, and Miles Simon.
\newblock Optimal curvature estimates for homogeneous {R}icci flows.
\newblock {\em Int. Math. Res. Not. IMRN}, (14):4431--4468, 2019.

\bibitem{BohmWilking}
Christoph B\"{o}hm and Burkhard Wilking.
\newblock Manifolds with positive curvature operators are space forms.
\newblock {\em Ann. of Math. (2)}, 167(3):1079--1097, 2008.

\bibitem{MR4109899}
Theodora Bourni, Mat Langford, and Giuseppe Tinaglia.
\newblock On the existence of translating solutions of mean curvature flow in slab regions.
\newblock {\em Anal. PDE}, 13(4):1051--1072, 2020.

\bibitem{BLT1}
Theodora Bourni, Mat Langford, and Giuseppe Tinaglia.
\newblock Collapsing ancient solutions of mean curvature flow.
\newblock {\em J. Differential Geom.}, 119(2):187--219, 2021.

\bibitem{MR3997128}
Simon Brendle.
\newblock Ricci flow with surgery on manifolds with positive isotropic curvature.
\newblock {\em Ann. of Math. (2)}, 190(2):465--559, 2019.

\bibitem{MR4176064}
Simon Brendle.
\newblock Ancient solutions to the {R}icci flow in dimension 3.
\newblock {\em Acta Math.}, 225(1):1--102, 2020.

\bibitem{RFuniquenessHighDim}
Simon Brendle, Panagiota Daskalopoulos, Keaton Naff, and Natasa Sesum.
\newblock Uniqueness of compact ancient solutions to the higher-dimensional {R}icci flow.
\newblock {\em Journal f{\"u}r die reine und angewandte Mathematik (Crelles Journal)}, 2022.

\bibitem{MR4323639}
Simon Brendle, Panagiota Daskalopoulos, and Natasa Sesum.
\newblock Uniqueness of compact ancient solutions to three-dimensional {R}icci flow.
\newblock {\em Invent. Math.}, 226(2):579--651, 2021.

\bibitem{BrendleHuiskenSinestrari}
Simon Brendle, Gerhard Huisken, and Carlo Sinestrari.
\newblock Ancient solutions to the {R}icci flow with pinched curvature.
\newblock {\em Duke Math. J.}, 158(3):537--551, 2011.

\bibitem{Buttsworth}
Timothy Buttsworth.
\newblock {$SO(2)\times SO(3)$}-invariant {R}icci solitons and ancient flows on {$\Bbb S^4$}.
\newblock {\em J. Lond. Math. Soc. (2)}, 106(2):1098--1130, 2022.

\bibitem{Buz14}
Maria Buzano.
\newblock Ricci flow on homogeneous spaces with two isotropy summands.
\newblock {\em Ann. Global Anal. Geom.}, 45(1):25--45, 2014.

\bibitem{CaoSaloff-Coste}
Xiaodong Cao and Laurent Saloff-Coste.
\newblock Backward ricci flow on locally homogeneous 3-manifolds.
\newblock {\em Commun. Anal. Geom.}, (2):305--325, 2009.

\bibitem{10.4310/jdg/1246888488}
Bing-Long Chen.
\newblock {Strong uniqueness of the Ricci flow}.
\newblock {\em Journal of Differential Geometry}, 82(2):363 -- 382, 2009.

\bibitem{MR3207356}
Bing-Long Chen, Guoyi Xu, and Zhuhong Zhang.
\newblock Local pinching estimates in 3-dim {R}icci flow.
\newblock {\em Math. Res. Lett.}, 20(5):845--855, 2013.

\bibitem{MR2365237}
Bennett Chow, Sun-Chin Chu, David Glickenstein, Christine Guenther, James Isenberg, Tom Ivey, Dan Knopf, Peng Lu, Feng Luo, and Lei Ni.
\newblock {\em The {R}icci flow: techniques and applications. {P}art {II}}, volume 144 of {\em Mathematical Surveys and Monographs}.
\newblock American Mathematical Society, Providence, RI, 2008.
\newblock Analytic aspects.

\bibitem{MR2604955}
Bennett Chow, Sun-Chin Chu, David Glickenstein, Christine Guenther, James Isenberg, Tom Ivey, Dan Knopf, Peng Lu, Feng Luo, and Lei Ni.
\newblock {\em The {R}icci flow: techniques and applications. {P}art {III}. {G}eometric-analytic aspects}, volume 163 of {\em Mathematical Surveys and Monographs}.
\newblock American Mathematical Society, Providence, RI, 2010.

\bibitem{Chowetal}
Bennett Chow, Peng Lu, and Lei Ni.
\newblock {\em Hamilton's {R}icci flow}, volume~77 of {\em Graduate Studies in Mathematics}.
\newblock American Mathematical Society, Providence, RI; Science Press Beijing, New York, 2006.

\bibitem{MR2301253}
Sun-Chin Chu.
\newblock Type {II} ancient solutions to the {R}icci flow on surfaces.
\newblock {\em Comm. Anal. Geom.}, 15(1):195--215, 2007.

\bibitem{ColdingMinicozziWidth}
Tobias~H. Colding and William~P. Minicozzi, II.
\newblock Width and finite extinction time of {R}icci flow.
\newblock {\em Geom. Topol.}, 12(5):2537--2586, 2008.

\bibitem{MR2074874}
P.~Daskalopoulos and R.~Hamilton.
\newblock Geometric estimates for the logarithmic fast diffusion equation.
\newblock {\em Comm. Anal. Geom.}, 12(1-2):143--164, 2004.

\bibitem{MR2264733}
P.~Daskalopoulos and N.~Sesum.
\newblock Eternal solutions to the {R}icci flow on {$\Bbb R^2$}.
\newblock {\em Int. Math. Res. Not.}, pages Art. ID 83610, 20, 2006.

\bibitem{MR4204997}
Panagiota Daskalopoulos.
\newblock Ancient solutions to geometric flows.
\newblock In {\em First {C}ongress of {G}reek mathematicians}, De Gruyter Proc. Math., pages 29--49. De Gruyter, Berlin, [2020] \copyright 2020.

\bibitem{DHSRicci}
Panagiota Daskalopoulos, Richard Hamilton, and Natasa Sesum.
\newblock Classification of ancient compact solutions to the {R}icci flow on surfaces.
\newblock {\em J. Differential Geom.}, 91(2):171--214, 2012.

\bibitem{Fateev1}
V.~A. Fateev.
\newblock The duality between two-dimensional integrable field theories and sigma models.
\newblock {\em Phys. Lett. B}, 357(3):397--403, 1995.

\bibitem{Fateev2}
V.~A. Fateev.
\newblock The sigma model (dual) representation for a two-parameter family of integrable quantum field theories.
\newblock {\em Nuclear Phys. B}, 473(3):509--538, 1996.

\bibitem{Fateev0}
V.~A. Fateev, E.~Onofri, and Al.~B. Zamolodchikov.
\newblock Integrable deformations of the {${\rm O}(3)$} sigma model. {T}he sausage model.
\newblock {\em Nuclear Phys. B}, 406(3):521--565, 1993.

\bibitem{Gromov}
Mikhael Gromov.
\newblock Filling {R}iemannian manifolds.
\newblock {\em J. Differential Geom.}, 18(1):1--147, 1983.

\bibitem{MR862046}
Richard~S. Hamilton.
\newblock Four-manifolds with positive curvature operator.
\newblock {\em J. Differential Geom.}, 24(2):153--179, 1986.

\bibitem{MR954419}
Richard~S. Hamilton.
\newblock The {R}icci flow on surfaces.
\newblock In {\em Mathematics and general relativity ({S}anta {C}ruz, {CA}, 1986)}, volume~71 of {\em Contemp. Math.}, pages 237--262. Amer. Math. Soc., Providence, RI, 1988.

\bibitem{HamiltonHarnackRicci}
Richard~S. Hamilton.
\newblock The {H}arnack estimate for the {R}icci flow.
\newblock {\em J. Differential Geom.}, 37(1):225--243, 1993.

\bibitem{MR1714939}
Richard~S. Hamilton.
\newblock Non-singular solutions of the {R}icci flow on three-manifolds.
\newblock {\em Comm. Anal. Geom.}, 7(4):695--729, 1999.

\bibitem{Haslhofer4d}
Robert Haslhofer.
\newblock On $\kappa$-solutions and canonical neighborhoods in 4d {R}icci flow.
\newblock \href{https://arxiv.org/abs/2308.01448}{arXiv:2308.01448}.

\bibitem{MR49430}
I.~I. Hirschman, Jr.
\newblock A note on the heat equation.
\newblock {\em Duke Math. J.}, 19:487--492, 1952.

\bibitem{IJ92}
James Isenberg and Martin Jackson.
\newblock Ricci flow of locally homogeneous geometries on closed manifolds.
\newblock {\em J. Differential Geom.}, 35(3):723--741, 1992.

\bibitem{MR1249376}
Thomas Ivey.
\newblock Ricci solitons on compact three-manifolds.
\newblock {\em Differential Geom. Appl.}, 3(4):301--307, 1993.

\bibitem{MR1214546}
J.~R. King.
\newblock Exact polynomial solutions to some nonlinear diffusion equations.
\newblock {\em Phys. D}, 64(1-3):35--65, 1993.

\bibitem{Klin78}
Wilhelm Klingenberg.
\newblock {\em Lectures on closed geodesics}.
\newblock Grundlehren der Mathematischen Wissenschaften, Vol. 230. Springer-Verlag, Berlin-New York, 1978.

\bibitem{KrishnanPediconiSbiti}
Anusha Krishnan, Francesco Pediconi, and Sammy Sbiti.
\newblock Toral symmetries of collapsed ancient solutions to the homogeneous {R}icci flow.
\newblock Preprint, \href{https://arxiv.org/abs/2312.01469}{arXiv:2312.01469}.

\bibitem{Lai2}
Yi~Lai.
\newblock 3d flying wings for any asymptotic cones.
\newblock Preprint, \href{https://arxiv.org/abs/arXiv:2207.02714}{arXiv:2207.02714}.

\bibitem{Lai1}
Yi~Lai.
\newblock A family of 3{D} steady gradient solitons that are flying wings.
\newblock {\em J. Differential Geom.}, 126(1):297--328, 2024.

\bibitem{Lau13}
Jorge Lauret.
\newblock Ricci flow of homogeneous manifolds.
\newblock {\em Math. Z.}, 274(1-2):373--403, 2013.

\bibitem{ZHANG2007503}
Zhei lei Zhang.
\newblock Compact blow-up limits of finite time singularities of ricci flow are shrinking ricci solitons.
\newblock {\em Comptes Rendus Mathematique}, 345(9):503--506, 2007.

\bibitem{MR3689745}
Peng Lu and Y.~K. Wang.
\newblock Ancient solutions of the {R}icci flow on bundles.
\newblock {\em Adv. Math.}, 318:411--456, 2017.

\bibitem{MarquesNevesSurvey}
Fernando~C. Marques and Andr\'{e} Neves.
\newblock Topology of the space of cycles and existence of minimal varieties.
\newblock In {\em Surveys in differential geometry 2016. {A}dvances in geometry and mathematical physics}, volume~21 of {\em Surv. Differ. Geom.}, pages 165--177. Int. Press, Somerville, MA, 2016.

\bibitem{PediconiSbiti}
Francesco Pediconi and Sammy Sbiti.
\newblock Collapsed ancient solutions of the {R}icci flow on compact homogeneous spaces.
\newblock {\em Proceedings of the London Mathematical Society}, 125(5):1130--1151, 2022.

\bibitem{Perelman1}
Grisha Perelman.
\newblock The entropy formula for the {R}icci flow and its geometric applications.
\newblock \href{https://www.arXiv.org/abs/math/0211159}{arXiv:0211159}.

\bibitem{Perelman3}
Grisha Perelman.
\newblock {F}inite extinction time for the solutions to the {R}icci flow on certain three-manifolds.
\newblock \href{https://www.arXiv.org/abs/math/0307245}{arXiv:math/0307245}.

\bibitem{Perelman2}
Grisha Perelman.
\newblock {R}icci flow with surgery on three-manifolds.
\newblock \href{https://www.arXiv.org/abs/math/0303109}{arXiv:0303109}.

\bibitem{Petersen}
Peter Petersen.
\newblock {\em Riemannian geometry}, volume 171 of {\em Graduate Texts in Mathematics}.
\newblock Springer, Cham, third edition, 2016.

\bibitem{Pitts}
Jon~T. Pitts.
\newblock {\em Existence and regularity of minimal surfaces on {R}iemannian manifolds}, volume~27 of {\em Mathematical Notes}.
\newblock Princeton University Press, Princeton, N.J.; University of Tokyo Press, Tokyo, 1981.

\bibitem{Rosenau}
P.~Rosenau.
\newblock On fast and super-fast diffusion.
\newblock {\em Phys Rev Lett.}, 74(7):1056--1059, 1995.

\bibitem{Sb22}
Sammy Sbiti.
\newblock On the {R}icci flow of homogeneous metrics on spheres.
\newblock {\em Ann. Global Anal. Geom.}, 61(3):499--517, 2022.

\bibitem{MR1001277}
Wan-Xiong Shi.
\newblock Deforming the metric on complete {R}iemannian manifolds.
\newblock {\em J. Differential Geom.}, 30(1):223--301, 1989.

\end{thebibliography}

\end{document}